\documentclass[12pt]{article}

\usepackage{amsfonts}
\usepackage{xr}
\usepackage{color}
\usepackage{graphicx,epsfig,paralist}
\usepackage{amsmath,amsthm,bbm}
\usepackage{hyperref}
\usepackage[active]{srcltx} %%allows you to jump from your editor
 \usepackage{color,soul} %%allows you to use the \st{...} command, which
 \setstcolor{red} %%the line of \st is red

\def\2{C^{1,2}(\R\times\R^N)}

\def\e{\epsilon}

\def\L{\mathcal{L}}
\def\Z{\mathbb{Z}}
\def\R{\mathbb{R}}
\def\N{\mathbb{N}}

\def\tilde{\widetilde}
\def\epsilon{\varepsilon}
\let\vp=\varphi
\def\trait (#1) (#2) (#3){\vrule width #1pt height #2pt depth #3pt}
\def\fin{\hfill\trait (0.1) (5) (0) \trait (5) (0.1) (0) \kern-5pt \trait (5) (5)
(-4.9) \trait (0.1) (5) (0)}
\def\1{\mathbbm{1}}

\newtheorem{thm}{Theorem}[section]

\newtheorem{lem}[thm]{Lemma}
\newtheorem{prop}[thm]{Proposition}

\theoremstyle{definition}
\newtheorem{defi}[thm]{Definition}
\newtheorem{rmk}{Remark}
\newtheorem{hyp}{Hypothesis}

\def\lm#1{\left\lfloor{#1}\right\rfloor}
\def\um#1{\left\lceil{#1}\right\rceil}
\def\l{\lambda_{1}(\L,\R)}

\begin{document}
\title{Generalized transition fronts for one-dimensional almost periodic Fisher-KPP equations.}

\author{Gr\'egoire Nadin 
\thanks{CNRS, UMR 7598, Laboratoire Jacques-Louis Lions, F-75005 Paris, France}  
\and Luca Rossi 
\thanks{ Dipartimento di Matematica,  Universit\`a di Padova, via Trieste 
63, 35121 Padova, Italy}
}
\maketitle

\date{}
\relpenalty=10000
\binoppenalty=10000

\begin{abstract} 
This paper investigates the existence of generalized
transition fronts for Fisher-KPP equations in one-dimensional, almost periodic
media. Assuming that the linearized elliptic operator near the unstable steady 
state admits an almost periodic eigenfunction,
we show that such fronts exist if and only if their average speed is
above an explicit threshold. This hypothesis is satisfied in particular when 
the reaction term does not depend on $x$ or (in some cases) is small 
enough.
Moreover, except for the threshold case, the
fronts we construct and their speeds are almost
periodic, in a sense.

When our hypothesis is no longer satisfied, such generalized transition 
fronts still exist for an interval of average speeds, with explicit bounds.

Our proof relies on the construction of sub and super
solutions based on an accurate analysis of the properties of the generalized
principal eigenvalues. 

% This paper investigates the existence of generalized 
% transition fronts for Fisher-KPP equations in one-dimensional, almost periodic 
% media. We show that such fronts exist if and only if their average speed is 
% above a suitable threshold. Moreover, except for the threshold case, the 
% fronts we construct and their speeds are almost 
% periodic, in a sense. Our proof relies on the construction of sub and super 
% solutions based on an accurate analysis of the properties of the generalized 
% principal eigenvalues. 
% When the reaction term also depends on the space variable almost periodically, 
% our result still holds but under some additional condition on the generalized 
% principal eigenvalues. 
\end{abstract}

\noindent {\bf Key-words:} generalized transition fronts, almost periodic,  
Fisher-KPP equation, generalized principal eigenvalues. 

\medskip

\noindent {\bf AMS classification:} 35B08, 35B15, 35B40, 35B50, 35C07, 35K57, 47E05

\medskip

The research leading to these results has received funding from the 
European Research Council
under the European Union's Seventh Framework Programme (FP/2007-2013) / ERC 
Grant
Agreement n.321186 - ReaDi -Reaction-Diffusion Equations, Propagation and 
Modelling.

\section{Introduction and main results}

\subsection{Introduction}

We are concerned with one-dimensional Fisher-KPP equations of the type
\begin{equation} \label{eqprinc}
 u_t - \big(a(x)u_x\big)_x = c(x) u(1-u), \quad t\in\R,\ x\in\R,
\end{equation}
with coefficients $a\in \mathcal{C}^1(\R)$, $c\in 
\mathcal{C}(\R)$ satisfying the following assumptions:
$$\inf_\R a >0,\qquad \inf_{\R}c>0,\qquad a,a',c\ \text{ are almost
periodic}.$$
These hypotheses will always be understood throughout the paper without further
mention.
We use Bochner's definition of almost periodic functions. 

\begin{defi} \cite{Bochner}\label{defalmostper} A function $a:\R\to\R$ is
\emph{almost periodic} (a.p.~in the sequel) if from any sequence
$(x_n)_{n\in\N}$ in $\R$ one can extract a subsequence
$(x_{n_k})_{k\in\N}$ such that $a(x_{n_k}+x)$ converges uniformly
in $x\in\R$. 

A function $U=U(z,x)$ is a.p.~in $x$ uniformly with respect to 
$z\in\R$ if from any sequence
$(x_n)_{n\in\N}$ in $\R$ one can extract a subsequence
$(x_{n_k})_{k\in\N}$ such that $U(z,x_{n_k}+x)$ converges uniformly
in $(z,x)\in\R\times\R$. 

\end{defi}

When $a$ and $c$ does not depend on $x$, this equation was investigated in the 
pioneering  papers of Kolmogorov, Petrovsky and Piskunov \cite{KPP} 
and Fisher \cite{Fisher}, who addressed the issue of the existence of 
special solutions, called {\em travelling fronts}: positive solutions 
of (\ref{eqprinc}) of the form $u(t,x)=U(x-wt)$ satisfying $U(-\infty)=1$ and 
$U(+\infty)=0$, with unknowns $U$ (the profile) and $w$ (the speed). Such 
solutions exist if and only if $w\geq 2\sqrt{ac}=:w^*$. Moreover, 
the front with minimal speed $w^*$ attracts, in a sense, the solutions of the 
Cauchy problem associated with the initial
datum $\1_{(-\infty,0)}$, see \cite{KPP}. 
The existence of travelling fronts was extended to more
general types of nonlinearities and multi-dimensional equations by Aronson and Weinberger \cite{AronsonWeinberger}.

When $a$ and $c$ are positive functions of $x$ which are periodic, 
with the same period, such a heterogeneity should be taken into account in the 
definition of the front, giving rise to the notion of {\em pulsating 
travelling front}: a positive solution of the form 
$u(t,x)=U(x-wt,x)$, where $U=U(z,x)$ is periodic in $x$, $U(-\infty,x)=1$ and 
$U(+\infty,x)=0$ uniformly in $x$. 
Analogously to the case of $x$-independent coefficients, pulsating 
travelling fronts exist if and only if $w\geq w^{*}$, where now $w^*>0$ is 
expressed in terms of the periodic principal eigenvalues of some linear operators, see 
\cite{BerestyckiHamel}. This result holds true for more general types of 
nonlinearities and in multidimensional media \cite{BerestyckiHamel, BHR2, 
Xinperiodic}, as well as when the coefficients are not only periodic in $x$ but 
also in $t$ \cite{Nadinptf, NolenRuddXin} . 

An increasing attention has been paid to the case of general heterogeneous coefficients in the 2000's. 
A generalization of the notion of travelling fronts has been given by Berestycki
and Hamel \cite{BerestyckiHamelgtw,BerestyckiHamelgtw2}.

\begin{defi}\label{def-gtf}\cite{BerestyckiHamelgtw, BerestyckiHamelgtw2}
A {\em generalized transition front} of equation (\ref{eqprinc}) is a 
time-global solution $u$ for which there exists a function $X:\R\to\R$ such 
that 
\begin{equation}\label{limits}
\lim_{x\to-\infty} u(t,x+X(t))=1, \quad
\lim_{x\to+\infty}u(t,x+X(t))=0,\quad \hbox{uniformly in } t\in\R. 
\end{equation}
We say that $u$ has an {\em average speed} $w\in\R$ if it holds true that
\[
\lim_{t-s\to+\infty}
\frac{X(t)-X(s)}{t-s}=w.
\]
% \end{equation}
\end{defi}

Travelling fronts and pulsating travelling fronts are particular 
instances of generalized transition fronts, with $X(t)=wt$. 

The existence of generalized transition fronts for 1-dimensional 
heterogeneous reaction-diffusion equations with ignition-type nonlinearities 
has been derived in \cite{NolenRyzhik, MelletRoquejoffreSire}, and their 
stability in \cite{MelletNolenRoquejoffreRyzhik}. It is also known for bistable 
time-heterogeneous equations \cite{Shenbistable}. 
For Fisher-KPP equations such as (\ref{eqprinc}), generalized transition waves do not exist
in general when the reaction term $c=c(x)$ is heterogeneous \cite{NRRZ}, but they do if only the diffusion term $a=a(x)$ is heterogeneous \cite{ZlatosKPP} or if the coefficients $a=a(t)$ and $c=c(t)$ are time heterogeneous \cite{NR1} (see also \cite{Shenue} dealing with time uniquely ergodic coefficients). 
This leads to an alternative notion of
{\em critical travelling wave} given by the first author in \cite{NadinCTW} for
1-dimensional equations. 
The situation is much more complicated in multi-dimensional media. Some 
conditions have been provided guaranteeing the existence of generalized 
transition fronts 
for partially periodic equations of ignition-type \cite{Zlatosdisordered} or 
for space-time heterogeneous Fisher-KPP equations under a space periodicity 
assumption \cite{NR2, RossiRyzhik, Shenueper}. However, for general 
heterogeneous ignition-type equations in $\R^{N}$, $N\geq 3$, generalized 
transition fronts do not exist in general \cite{Zlatoscex}, unlike the 
1-dimensional framework. 

In order to go further in the investigation of fronts in heterogeneous 
Fisher-KPP equations, we are thus lead to consider specific classes of 
heterogeneity since generalized transition fronts do not exist in general for 
such equations. The aim of the present paper is to understand the case of 
spatial a.p.~coefficients in dimension $1$. 
The only related results in the literature we are aware of are Shen's 
\cite{Shenap1, Shenap2}, concerning bistable equations with a time a.p.~reaction 
term, a typical example being 
$$u_{t} =u_{xx}+ u(1-u) \big( u - \theta (t)\big)$$
where $\theta$ is an a.p.~function such that $0<\theta<1$ in $\R$. 
This is quite a different problem, but we note that the type of solutions 
obtained by Shen are indeed similar to the one of Theorem~\ref{thm:gtf} below. 
However, the construction is very different: Shen uses a 
stability approach, which is well-fitted for bistable equations, combined with a 
parabolic zero number argument, while in the present paper we use a sub and 
super solutions method. We also mention here the work of Lou and Chen \cite{LouChen}
on curved travelling fronts for the curvature flow equation with space a.p. coefficients, which is known to be an asymptotic limit of bistable equations. 

Our approach here is inspired by the case of periodic coefficients 
(see \cite{Nadinptf} for example), where the sub and super solutions are 
constructed starting from exponential solutions of the linearized equation 
around the unstable steady state $u\equiv 0$. In the periodic case, this 
naturally leads to consider some eigenvalue problems for linear operators. The 
main difficulty here is that it is not clear whether an elliptic operator with 
a.p.~coefficients admits classical eigenvalues and, even when it 
does, there could exist multiple {\em bounded} eigenfunctions and they might 
not be 
a.p.~\cite{Rossiap}. We will thus make use of particular 
exponentially decreasing solutions of linear problems (see Proposition 
\ref{prop-phigamma}) and of generalized principal eigenvalues (defined by 
\eqref{eq:lambda1}-\eqref{eq:lambda12} below).

%--------------------------------------------------------------------------

\subsection{Statement of the results}

Our characterization of the minimal average speed of generalized transition 
fronts involves, as in periodic media \cite{BHR2, Nadinptf}, the 
eigenfunctions of the linearized operator near the unstable steady state 
$u\equiv 0$, that is,
\begin{equation} \label{def-La}
 \L \phi := \big(a(x) \phi'\big)' + c(x)\phi.
\end{equation}
Unlike in the periodic case, here we do not dispose of the compactness 
properties that allow one to define the eigenvalues in a classical sense. In 
particular, we cannot apply the Krein-Rutman theorem providing the principal 
eigenvalue. Thus, we consider the following notion of {\em generalized
principal eigenvalue}, introduced by Berestycki, Nirenberg and Varadhan in 
\cite{BNV}:
\begin{equation}\label{l1}
 \lambda_1:= \inf \{ \lambda\in\R, \ \exists \phi \in\mathcal{C}^2 
(\R),\ \phi>0,\ \L\phi \leq \lambda \phi \ \hbox{ in } \R\}.
\end{equation}
An equivalent definition was previously given by Agmon \cite{A1}
for divergence form operators on Riemannian
manifolds and, for general operators, by Nussbaum and Pinchover 
\cite{NP}, building on a result by Protter and Weinberger 
\cite{Max2}.
In the sequel, we will also make use of a characterization through a Rayleigh
quotient derived in~\cite{A1}, which readily implies that 
$\lambda_{1}$ 
is the infimum of the spectrum of $-\L$ on $L^{2}(\R)$, and also yields 
$\lambda_{1}\geq\inf c$ (see Proposition~\ref{lem-caraclambda1} below). 
In particular, under the standing assumptions, there holds
$$\lambda_1>0.$$

Besides the almost periodicity of the coefficients, we will assume the 
following.

\begin{hyp}\label{hyp}
 The operator $\L$ admits an a.p.~positive eigenfunction, that is, there exists 
an a.p.~positive function $\vp_{1} \in \mathcal{C}^{2}(\R)$ such that $\L 
\vp_{1}
= \lambda_{1} \vp_{1}$ in $\R$. 
\end{hyp}

The fact that the eigenvalue associated with an a.p.~positive eigenfunction is 
necessarily the generalized principal eigenvalue $\lambda_{1}$ follows from 
a Bloch-type property (see 
\cite[Theorem~1.7]{BerRossipreprint}). 
In the periodic case, Hypothesis \ref{hyp} trivially holds with $\vp_1$ equal 
to 
the periodic principal eigenvalue of $\L$.
We will see in Section \ref{sec:hyp} that Hypothesis \ref{hyp} is equivalent to 
require the existence of a bounded eigenfunction with positive infimum.
Let us point out that, for arbitrary operators, there always exists 
a positive solution of $\L \vp=
\lambda\vp_{1}$ in $\R$ \cite{A1,BerRossipreprint}, but it might not be a.p.,
bounded, nor with a positive infimum.

Checking Hypothesis \ref{hyp} is actually a difficult task in general. 
Sorets and Spencer \cite{SoretsSpencer} showed that, when $a\equiv 1$ and $c(x)
= K \big(\cos (2\pi x) + \cos(2\pi\alpha x)\big)$ with $\alpha \notin \mathbb{Q}$ and $K$
large enough, the Lyapounov exponent of 
$ \phi (n+1)-2\phi (n)+\phi (n-1)+c(n) \phi (n)= \lambda \phi(n)$
in $\Z$ is 
strictly positive, for any $\lambda\in\R$. This implies, through Ruelle-Oseledec's theorem, that 
any solution of this equation should either blow up or decay to zero 
exponentially, contradicting a possible almost periodicity. This example is 
especially striking because the above function~$c$ is quasiperiodic.
It involves the discrete Laplace operator, but we believe that such 
a phenomenon could also occur for the continuous equation $\L 
\phi = \lambda_{1}\phi$. Actually, \cite{SoretsSpencer} also deals with the
continuous Laplace operator, but excludes in this framework some eigenvalues and we were not able 
to determine if the principal eigenvalue $\lambda_{1}$ is among them.

On the other hand, Hypothesis \ref{hyp} holds in two relevant 
cases, where generalized transition fronts were not previously obtained:
\begin{enumerate}
 \item $c$ is constant.
 \item $a$ and $c$ are {\em quasiperiodic} and their periods satisfy the
non-degeneracy diophantine condition (\ref{dioph}) below. 
 \end{enumerate}
Indeed, in the first case Hypothesis \ref{hyp} holds because $\L$ admits 
constant eigenfunctions. The second 
case follows from a result by Kozlov \cite{Kozlov}, that we recall in 
Section \ref{sec:hyp}. 
The optimality of Hypothesis \ref{hyp}, and of a weaker one based on 
the theory of critical operators, will be discussed in Section \ref{sec:hyp}

\bigskip

We will need the following description of the other 
eigenlevels.

\begin{prop} \label{prop-phigamma} 
The following properties hold for all $\gamma>\lambda_{1}$:
\begin{enumerate}[(i)]
 \item There exists a unique positive 
solution $\phi_\gamma \in \mathcal{C}^2 (\R)$ of 
\begin{equation}\label{eq-prop-phigamma}
\L\phi_{\gamma}= \gamma \phi_{\gamma} \ \hbox{
in } \R, \quad \phi_{\gamma}(0)=1, \quad \lim_{x\to +\infty} 
\phi_{\gamma}(x)=0. \end{equation}
\item There exists the limit $\mu (\gamma) := 
-\lim_{x\to \pm\infty} \frac{1}{x}\ln \phi_\gamma (x)>0$.
\end{enumerate}
\end{prop}

We are now in position to state our main result. 

% As a corollary, we obtain the existence of generalized transition fronts. 
\begin{thm}\label{thm:gtf}
Let $\mu(\gamma)$ be as in Proposition \ref{prop-phigamma}. Then, 
% $\lambda_1>0$ and, 
under Hypothesis \ref{hyp}, there exists 
$$w^*:= \min_{\gamma>\lambda_1}\, \frac{\gamma}{\mu (\gamma)}\,>0$$
and the following properties hold:
\begin{enumerate}[(i)]
 \item For all $w\geq w^*$ there exists a time-increasing generalized 
transition front with average speed $w$; 
for $w>w^*$, the front can be written as $u(t,x) = U 
(\int_0^x\sigma -t,x)$, where 
$\sigma\in\mathcal{C} (\R)$ is a.p.~and has average $1/w$ and 
$U=U(z,x)$
is a.p.~in $x$ uniformly in $z\in\R$.

 \item There are no generalized transition fronts with average speed $w<w^*$.
\end{enumerate}
\end{thm}

Roughly speaking, the profile $U$ and the function $\sigma$ expressing the 
front in the supercritical case $w>w^*$ inherit the almost periodicity of the 
coefficients. Whether or not an analogous property holds true in the 
case $w=w^*$ is left as open question. 
% The 
% front provided by Theorem \ref{thm:gtf} part (ii) is a good candidate but we 
% were 
% not able to go further than the limits (\ref{eq:CTW}) (note that in order to 
% conclude that it has average speed equal to $w^*$, the $\inf$ and $\liminf$ in 
% the second limit should be switched). 
Moreover, in such case, the front being constructed as a critical travelling 
wave, in the 
sense of \cite{NadinCTW}, we conjecture that it is attractive for 
Heaviside type initial data (see the discussion about this topic in 
\cite{NadinCTW}).  

We derive the non-existence result as a consequence of an estimate of the 
propagation of the interface $X(t)$ of generalized transition fronts which is
valid even without assuming the existence of an
average speed (see Proposition \ref{prop:nonexistence} below). 
\\

% \subsection{Generalization to heterogeneous reaction terms $c(x)$}

Hypothesis \ref{hyp} allows us to prove the following key property: 
\begin{equation}\label{mu=0}
\lim_{\gamma\searrow \lambda_1}\mu(\gamma)=0.
\end{equation}
This limit is always well-defined since $\gamma\mapsto \mu (\gamma)$ will be proved to be nondecreasing and nonnegative (see Lemma \ref{lem-convexitymu}).  If Hypothesis \ref{hyp} fails, then exponentially localized eigenfunctions may arise \cite{SoretsSpencer}, meaning that this limit might be positive. 

If we drop Hypothesis \ref{hyp} we are still able to obtain a partial existence
result.

\begin{thm}\label{thm-apgen} 
Let $\mu (\gamma)$ be as in Proposition \ref{prop-phigamma} and set 
$$
w^*:=\inf_{\gamma>\lambda_1} \displaystyle\frac{\gamma}{\mu 
(\gamma)},\qquad
\underline{w}:=\frac{\lambda_1}{\underline{\mu}},
\quad\text{where}\quad
\underline{\mu}:= \lim_{\gamma\searrow \lambda_1} \mu (\gamma).$$
The following properties hold:
\begin{enumerate}[(i)]
 \item If $w^*<\underline{w}$ then for all $w\in [w^*,\underline{w})$, there exists a time-increasing generalized 
transition front with average speed $w$; 
for $w>w^*$, the front can be written as $u(t,x) = U 
(\int_0^x\sigma -t,x)$, where 
$\sigma\in\mathcal{C} (\R)$ is a.p.~and has average $1/w$ and 
$U=U(z,x)$
is a.p.~in $x$ uniformly in $z\in\R$.

 \item There are no generalized transition fronts with average speed $w<w^*$.
\end{enumerate}
\end{thm}

Since Hypothesis \ref{hyp} entails \eqref{mu=0}, which in turn yields 
$\underline w=+\infty$, Theorem 
\ref{thm:gtf} will follow from Theorem \ref{thm-apgen}, after showing that $w^*$ 
is attained in that case, see Remark \ref{rmk:hyp}.
We will also show that \eqref{mu=0} holds when $c$ is constant using an 
alternative direct argument, see Remark \ref{rmk:c}.

The above results leave open several interesting problems:
\begin{itemize}
\item Do they extend to multi-dimensional equations? The construction of the 
$\phi_{\gamma}$'s in Proposition \ref{prop-phigamma} strongly relies on 1D 
arguments, but are there analogous solutions in higher dimensions? 

\item Is it possible to construct a rigorous example where $\underline{\mu}>0$ and, in this case, are there generalized transition fronts with average speed $w>\underline{w}$? 

\item What are the properties of the critical front with speed $w=w^{*}$? Does it have a.p. profile? Is it attractive, in a sense, for the Cauchy problem? 
\end{itemize}

We conclude with a small lemma ensuring that the 
sufficient condition for the existence result in Theorem \ref{thm-apgen}(i) is 
fulfilled up to a constant perturbation of the zero order term.

\begin{lem} \label{lem:nonempty}
% Assume that $a$, $a'$ and $\tilde c$ are a.p.~continuous functions 
% such that $a$ and $\tilde c$ have a positive infimum. 
For $c_0$ large enough 
the quantities $\underline{w}$ and 
$w^{*}$ provided by Theorem \ref{thm-apgen} for the equation \eqref{eqprinc} 
with $c(x)$ replaced by $c(x)+c_{0}$ satisfy $w^*<\underline{w}$. 
\end{lem}

%-------------------------------

\subsection{Optimality of Hypothesis \ref{hyp}}
\label{sec:hyp}

In this section we discuss Hypothesis \ref{hyp} and a generalization. This
involves the notion of {\em limit operator} associated with $\L$, that is, an
operator of the type
$$\L^*\phi:=(a^*(x)\phi')'+ c^*(x)\phi,$$
where $a^*(x)$ and $c^*(x)$ are the limits as $n\to\infty$ of
$a(x + x_n)$ and $c(x + x_n)$, for some sequence
$(x_n)_{n\in\N}$ in $\R$.
We know from \cite[Lemma 5.6]{BerestyckiHamelRossi} that, being $a$ and $c$ 
a.p., the
generalized
principal eigenvalue of any limit operator $\L^*$ coincides with 
the one of $\L\,$: $\lambda_1$. 

\begin{hyp}\label{hyp-crit}
For any limit operator $\L^*$, the operator $\L^*-\lambda_1$
is {\em critical}, that is, the space of 
positive solutions to 
\begin{equation}\label{pep}
\L^{*}\vp=\lambda_{1}\vp\quad\text{in }\R 
\end{equation} 
has dimension 1. 
\end{hyp}
Actually, the definition of criticality is related to the
non-existence of a Green function (see \cite{Pinch88}) but in dimension 1 it is 
known from \cite{A1}, \cite[Appendix~1]{Murata}
that it is equivalent to the validity of the positive Liouville property stated
in Hypothesis~\ref{hyp-crit} (see also \cite[Theorem 4.3.4 and Proposition 
5.1.3]{Pinsky}). 
We will make use of the theory of critical operators
only once, in order to show that \eqref{mu=0} holds under Hypothesis 
\ref{hyp-crit}, see
Proposition \ref{prop:underlinemu}. It follows that \eqref{mu=0} holds under 
Hypothesis \ref{hyp}, which is more restrictive than
Hypothesis \ref{hyp-crit} owing to the following.

\begin{prop} \label{prop:hypimply}
Let $\L$ be a self-adjoint operator on $\R$ which admits a positive bounded 
eigenfunction $\vp$. Then, the associated eigenvalue is the generalized 
principal eigenvalue $\lambda_1$ given by 
\eqref{l1}, and there holds:
\begin{enumerate}
\item $\L-\lambda_1$ is critical;
 
\item if in addition $\inf\vp>0$ 
%\footnote{ Pinchover suggere que $\inf\vp>0$ pourrait peut-etre etre affaibli
%par une ``non-integrability condition'' dans \cite{Murata}. C'est tres interessant l'article de Murata. Apres j'ai essaye de prouver rapidement cette conjecture et c'est pas evident du tout. On garde ca pour plus tard? }
then $\L^*-\lambda_1$
is critical for any limit operator $\L^*$; 

\item 
% positive bounded eigenfunction with positive infimum is equivalent to the
% existence of an a.p.~positive eigenfunction.
$\inf\vp>0$ if and only if $\vp$ is a.p.
\end{enumerate}
\end{prop}

The above property $1$ could be obtained as a consequence of one of the 
following stronger results: \cite[Theorem A.9]{Murata} or \cite[Theorem 
1.7]{Pinch07}.
We present below a different, simple direct proof. The almost periodicity of 
the 
coefficients is only needed in $3$ (boundedness and uniform continuity would be 
enough to guarantee the existence of limit operators so that $2$ makes 
sense).  
We point out that in \cite{Murata} the hypothesis~$\vp$ bounded is 
replaced by an integrability condition which holds if $\vp$ grows at most as 
$\sqrt{|x|}$. 
However, unlike boundedness, such more general condition cannot be 
exploited to get informations about limit operators.
% In the proof below we use a different, direct argument.
% We prove the result with a direct approach, without using \cite{Murata}.

\begin{proof}
Firstly, we know from \cite[Theorem 1.7]{BerRossipreprint} that the existence 
of the positive bounded eigenfunction $\vp$ implies that the associated 
eigenvalue is necessarily $\lambda_1$. 

1. Let $\psi$ be a nontrivial solution to $\L\psi = \lambda_1 \psi$. We can 
assume without loss
of generality that $\psi(0)\neq0$. Let us normalize $\vp,\psi$ 
by $\vp(0)=\psi(0)=1$.
Using equation \eqref{pep} for both $\vp$ and $\psi$ we get, for all $x\in\R$,
$$\int_0^x (a\vp')'\psi=\int_0^x(a\psi')'\vp,$$
from which we deduce
$$a(x)(\vp'\psi-\psi'\vp)(x)=a(0)(\vp'-\psi')(0),$$
or equivalently
$$(\psi/\vp)'(x)=\frac{a(0)(\psi'-\vp')(0)}{a(x)\vp^2(x)}.$$
It follows that  if $(\psi'-\vp')(0)$ has a sign then the function
$(\psi/\vp)'$ has everywhere the same sign and it is bounded
away from $0$ since $\vp$ is bounded. This implies that  $\psi/\vp$ changes sign on $\R$, that is,
$\psi$ changes sign. As a consequence, $\vp$ is the unique positive
eigenfunction of $\L$ with eigenvalue $\lambda_1$, up to a scalar multiple, 
meaning that $\L-\lambda_1$ is critical. 

2. Suppose that $\inf\vp>0$. Let $\L^*$ be a limit operator defined through a 
sequence of translations $(x_n)_{n\in\N}$. The functions $\vp(\cdot+x_n)$
converge locally 
uniformly (up to subsequences) to an eigenfunction of $\L^*$ with 
eigenvalue $\lambda_1$. Moreover, because $\vp$ is bounded and has 
positive infimum, the same is true for such eigenfunction. Statement 1 of the
proposition eventually implies that 
$\L^*-\lambda_1$ is critical.

3. An a.p.~positive eigenfunction $\vp$ is necessarily bounded and with 
positive infimum. Indeed, boundedness immediately follows from almost 
periodicity, while, supposing that $\inf\vp=0$, one would 
readily obtain the contradiction $\vp\equiv0$ by applying the strong maximum 
principle to a limit operator associated with a minimizing sequence for 
$\vp$ 
and then using the almost periodicity of $\vp$. Conversely, let $\vp$ be a 
bounded eigenfunction with positive infimum. We use the following 
characterization of a.p.~functions due to Bochner~\cite{Boc}:
a~function $g:\R\to\R$ is a.p.~if and only if from any pair of sequences
$(x_n)_{n\in\N}$, $(y_n)_{n\in\N}$ one can extract a common subsequence
$(x'_n)_{n\in\N}$, $(y'_n)_{n\in\N}$ such that
$$\forall 
x\in\R,\quad\lim_{n\to\infty}g(x+x'_n+y'_n)=\lim_{m\to\infty}\left(\lim_{
n\to\infty } g(x+x'_n+y'_m)\right).$$
Consider two sequences $(x_n)_{n\in\N}$, $(y_n)_{n\in\N}$ in $\R$.
By elliptic estimates, the limits 
$$\tilde\vp(x):=\lim_{n\to\infty}\vp(x+x_n),\quad
{\tilde\vp}^*(x):=\lim_{n\to\infty}\tilde\vp(x+y_n),\quad
\vp^*(x):=\lim_{n\to\infty}\vp(x+x_n+y_n),$$
exist (up to subsequences) locally uniformly in 
$x\in\R$, and, applying Bochner's characterization to the coefficients $a$, $c$,
we deduce that both $\tilde\vp^*$ and $\vp^*$ are eigenfunctions 
with eigenvalue~$\lambda_1$ of the limit operator $\L^*$ associated with (a subsequence of) the 
sequence of translations $(x_n+y_n)_{n\in\N}$.
The operator $\L^*-\lambda_1$ is critical by statement 2, which means that the positive functions $\tilde\vp^*$ and $\vp^*$ coincide up to a scalar multiple $\beta>0$. Namely
\begin{equation}\label{multiple}
\forall x\in\R,\quad
\lim_{m\to\infty}\left(\lim_{n\to\infty}\vp(x+x_n+y_m)\right)=
\beta\lim_{n\to\infty}\vp(x+x_n+y_n).
\end{equation}
If we show that $\beta=1$, we would infer from 
Bochner's characterization that $\vp$ is a.p. To do this, we first 
deduce from the fact that $\vp^*$, $\tilde\vp$ are obtained as limits of  translations of $\vp$ and that $\tilde\vp^*$ is a limit of translations of $\tilde\vp$ that
$$\inf\vp^*\geq\inf\vp,\qquad \sup\vp^*\leq\sup\vp,\qquad
\inf\tilde\vp^*\geq\inf\tilde\vp\geq\inf\vp,\qquad \sup\tilde\vp^*\leq\sup\tilde\vp\leq\sup\vp.$$
Now, we apply property \eqref{multiple} with $(x_n)_{n\in\N}$ replaced by $(x_n+y_n)_{n\in\N}$ and with $(y_n)_{n\in\N}$ replaced by $(-x_n-y_n)_{n\in\N}$. We infer the existence of some $\beta^*>0$ such that (up to subsequences)
$$\forall x\in\R,\quad
\lim_{m\to\infty}\vp^*(x-x_m-y_m)=\beta^*\vp(x).$$
This means that $\beta^*\vp$ is a limit of translations of $\vp^*$ and
therefore 
$$\beta^*\inf\vp\geq\inf\vp^*\geq\inf\vp,\qquad \beta^*\sup\vp\leq\sup\vp^*
\leq\sup\vp.$$
It follows that $\beta^*=1$ and $\inf\vp^*=\inf\vp$, $\sup\vp^*=\sup\vp$.
With analogous arguments, considering the pair of sequences $(x_n)_{n\in\N}$ and
$(-x_n)_{n\in\N}$ we derive $\inf\tilde\vp=\inf\vp$, $\sup\tilde\vp=\sup\vp$.
Finally, starting from the function $\tilde\vp$ 
(which satisfies the same type of eigenvalue problem with a.p.~coefficients as
$\vp$) and considering the sequences $(y_n)_{n\in\N}$ and $(-y_n)_{n\in\N}$, we
find that 
$\inf\tilde\vp^*=\inf\tilde\vp$, $\sup\tilde\vp^*=\sup\tilde\vp$.
Summing up, we have that $\inf\tilde\vp^*=\inf\vp^*$,
$\sup\tilde\vp^*=\sup\vp^*$, whence $\beta=1$.
\end{proof}

\begin{prop}[\cite{Kozlov}]\label{prop:Kozlov}
Let $M\geq 2$ and $\omega=(\omega_1,\dots,\omega_M) \in \R^{M}$ be such that 
\begin{equation} \label{dioph}\forall n\in \mathbb{Z}^{M}\backslash \{0\}, \quad
|n\cdot\omega |
\geq k |n|^{-\sigma} \quad \hbox{ for some }\; k, \sigma >0.\end{equation}
Assume that $a$ and $c$ are quasiperiodic, in the sense that there exist
two functions $\hat{a},\hat{c}\in\mathcal{C}(\R^{M},\R)$ with periods
$\Z^M$ which are $1-$periodic in all directions, and $(\omega_{1} ,..., \omega_{M})\in (0,\infty)^{M}$
 such that $a(x) = \hat{a}(\omega_{1} x,..., \omega_{M}x)$ and 
$c(x) = \hat{c}(\omega_{1} x,..., \omega_{M}x)$ for all $x\in \R$.
Then there exist $r=r(\sigma)$ and $\epsilon = \epsilon (\sigma)$ such that if
$\hat{a}, \hat{c}\in \mathcal{C}^{r}(\R^{M},\R)$ and
$\|\hat{c}\|_{\mathcal{C}^{r}}<\epsilon$, then 
Hypothesis \ref{hyp} is satisfied. 
\end{prop}

A typical example of a function $c$ satisfying these hypotheses is $c(x) =\epsilon\big( \cos (x)+ \cos (\sqrt{2}x)\big)$, with $\epsilon$ small enough. 

%----------------------------------------
%
%
%\subsection*{Organization of the paper}  
%We first deal in Section \ref{sec:linearized} with the linearized problem 
%around $u \equiv 0$ and with the properties of the function $\gamma\mapsto \mu 
%(\gamma)$. Then we construct our fronts in Section \ref{sec:construction}. 
%Section \ref{sec:properties} is devoted to the properties of the fronts: their 
%speed, the fact that they are generalized transition fronts and the 
%non-existence result for small speeds. 

%%%%%%%%%%%%%%%%%%%%%%%%%%%%%%%%%%%%%%%%%%%%%%%%%%%%%%%%%%%%%%%%%%%%%%%%%%%%%%

\section{Properties of the linearized problem}\label{sec:linearized}

We investigate now the properties of the eigenfunctions of the linearized 
operator $\L$ defined by \eqref{def-La}. For any (possibly unbounded) interval 
$I$ we define the generalized principal eigenvalue
%\footnote{ou bien on utilise \eqref{def-RayleighR} comme definition. Qu'est-ce 
%que tu en penses? C'est peut etre mieux pour coller aux notions des gens qui
%travaillent sur le spectre $L^{2}$.\\
%Je te laisse choisir. Une fois decid\'e il faudra bien controller dans la
%suite lorsqu'on fait reference a' la def ou bien a' la caracterisation.}
\begin{equation}\label{charI}
\lambda_1 (\L,I)=\inf \{ \lambda\in\R, \ \exists \phi\in\mathcal{C}^2(I),
\ \phi>0 \ \hbox{ in } I, \ \L \phi \leq \lambda \phi\text{ in }I\}.
\end{equation}
If $I$ is bounded it coincides with the classical principal eigenvalue.
In the case $I=\R$ the definition reduces to that of $\lambda_1$ 
in \eqref{l1}; we will sometimes use the notation $\l$ in order to 
avoid ambiguity.
The following characterizations hold without assuming $a$, $c$ to be 
almost periodic, but just bounded (see \cite{A1}, \cite{BerRossipreprint}):
\begin{equation} \label{def-RayleighR}
 \lambda_1 (\L,I)= \sup_{\varphi\in H^1_0 (I),\ 
\varphi\not\equiv0}\frac{\int_I \left( c(x) \varphi^2 -
a(x)(\varphi')^2 \right) dx}{\int_I \varphi^2 dx},
\end{equation}
\begin{equation}\label{charR}
\lambda_1 (\L,\R)=\lim_{R\to+\infty}\lambda_1 (\L, (-R,R)).
\end{equation}
From \eqref{charI} it follows that $\lambda_1 (\L,I)$ is nondecreasing 
with
respect to the inclusion of intervals $I$.
If $\L$ has a.p.~coefficients then $\lambda_1 (\L,\R)$ can also be
characterized through intervals invading only $\R_+$ (or $\R_-$).

\begin{prop}\label{lem-caraclambda1}
There holds
$$\lambda_1 (\L, (0,R))\nearrow\lambda_1 (\L,\R)\geq\inf c \quad 
\text{ as}\quad R\nearrow+\infty.$$
\end{prop}

\begin{proof} 
First, it is well-known that for all $R>0$, there exists a Dirichlet
principal eigenfunction, that is, $\varphi_R \in 
\mathcal{C}^2 ([0,R])$ 
such that $\varphi_R (0)=\varphi_R (R) = 0$, 
$\varphi_R >0$ in $(0,R)$ and $\big( a(x)\varphi_R'\big)'+ c(x) \varphi_R =
\lambda_1 (\L,
(0,R))\varphi_R $ in $(0,R)$. From (\ref{charI}), it follows that $R\mapsto
\lambda_1 (\L, (0,R))$ is nondecreasing and bounded from above by 
$\lambda_1 (\L,\R)$. Hence, one can define 
$$\lambda:= \lim_{R\to +\infty} \lambda_1 (\L, (0,R))\leq \lambda_1 (\L,\R).$$
The Harnack inequality, elliptic regularity estimates and 
a diagonal extraction imply that there exists a sequence
$(R_n)_n$ such that $R_n\to +\infty$ and 
the functions $(\psi_n)_n$ defined by
$\psi_n(x):=\varphi_{2R_n}(x+R_n)/\varphi_{2R_n}(R_n)$
converge to some $\psi_\infty$ in
$\mathcal{C}^1 (K)$ for all compact set $K \subset\R$. 
Since $a$ and $c$ are a.p., we can assume without loss of generality that
there exists $a^*,c^*\in\mathcal{C} (\R)$ such that $a(x+R_n)\to a^*(x)$ and
$c(x+R_n)\to c^*(x)$ as $n\to\infty$ uniformly in $x\in\R$.
It follows that $\psi_\infty$ is a weak solution of 
% $$\varphi_\infty''+ c(x) \varphi_\infty = \lambda\varphi_\infty \hbox{ in }
% (0,\infty), \ \varphi_\infty (1) = 1, \ \varphi_\infty (0)=0, \
% \varphi_\infty\geq 0 \hbox{ in } (0,\infty).$$
% Define for all $n\in\N$, $\psi_n (x) = \varphi_\infty (x+n)/\varphi_\infty
% (n)$. As $c$ is almost periodic, there exists an almost periodic
% $c^*\in\mathcal{C} (\R)$ 
% and an increasing function $\theta :\N \to \N$ so that $c(x+\theta (n)) \to
% c^* (x)$ as $n\to +\infty$ uniformly in $x$. 
% Similar arguments as above yield that one can assume, up to extraction, 
% that the sequence $(\psi_{\theta(n)})_n$ converges to a function $\psi_\infty$
% which satisfies
$$\L^* \psi_\infty = \lambda\psi_\infty \ \hbox{ in } \R, \quad
\psi_\infty (0) = 1, \quad \psi_\infty\geq 0 \ \hbox{ in } \R,$$
where $\mathcal{L}^*$ is the limit operator defined by
$\mathcal{L}^*\phi:=\big( a^*(x)\phi'\big)'+c^*(x)\phi$.
It follows that $\psi_\infty\in \mathcal{C}^2(\R)$ and, by the strong maximum 
principle,
that $\psi_\infty>0$ in $\R$. Then the characterization (\ref{charI}) yields
$\lambda_1 (\L^*,\R)\leq\lambda$. But since $\lambda_1 
(\L^*,\R)=\lambda_1(\L,\R)$ by \cite[Lemma
5.6]{BerestyckiHamelRossi} because $\L$ is a.p., we conclude that $\lambda =
\lambda_1 (\L,\R)$. 

Finally, taking $\varphi(x)=\cos(\frac\pi{2R}x)$ in the characterization
\eqref{def-RayleighR} we deduce that, as $R\to\infty$,
$\lambda_1(\L,(-R,R))\geq \inf
c+O(R^{-2})$, whence, by \eqref{charR}, 
$\lambda_1(\L,\R)\geq \inf c$.
\end{proof}

In the sequel we will make frequent use of the following technical lemma, which 
is an immediate consequence of Lemma 2.2 in \cite{Nolen}. The latter 
was proved for $a\equiv 1$ but the reader could easily check that it holds 
true for an elliptic diffusion term $a$ satisfying the hypotheses of the
present paper. 

% \begin{lem} \label{lem:Nolen}
% Let $\phi\in \mathcal{C}([x_0,+\infty))$ be a generalized 
% subsolution of 
% \begin{equation*}
% \big( a(x)\phi'\big)' + 
% (c(x)-\gamma) \phi  = 0 \ \hbox{ in }
% (x_0,+\infty), \quad \lim_{x\to+\infty}\phi(x)=0. 
% \end{equation*} 
% Then, for any $\e>0$, there exists $C>0$ such that
% \begin{equation*}
% \forall x>x_0,\quad
% \phi(x) \leq C \big(\phi(x_0)\big)_{+} e^{-\left(\sqrt{\gamma
% -\lambda_1(\L,\R)} -\e\right) x}.
% \end{equation*}
% \end{lem}

\begin{lem} \label{lem:Nolen}
For all $\gamma>\lambda_{1}(\L,\R)$, $x_0\in\R$ and $\e>0$, there exists $C>0$ 
such that any  
generalized subsolution $\phi\in \mathcal{C}([x_0,+\infty))$ of 
\begin{equation*}
\L\phi  =\gamma\phi \ \hbox{ in }
(x_0,+\infty), \quad \lim_{x\to+\infty}\phi(x)=0,
\end{equation*} 
satisfies
\begin{equation*}
\phi(x) \leq C \big(\max\{\phi(x_0),0\}\big) e^{-\left(\sqrt{\gamma
-\lambda_1(\L,\R)} -\e\right) x}.
\end{equation*}
\end{lem}

% Owing to the second part of Proposition  
% \ref{lem-caraclambda1}, Proposition \ref{prop-phigamma} is a consequence of the 
% following.
% 
% \begin{prop} \label{prop-phigamma2}
% % Properties 1-2) of Proposition \ref{prop-phigamma} hold for the operator 
% % $\mathcal{L}$ for all $\gamma>\lambda_{1}(\L,\R)$.
% The following properties hold for all $\gamma>\lambda_{1}(\L,\R)$:
% \begin{enumerate}
%  \item there exists a unique positive 
% solution $\phi_\gamma \in \mathcal{C}^2 (\R)$ of 
% \begin{equation}\label{eq-prop-phigamma}
% \L\phi_{\gamma}= \gamma \phi_{\gamma} \ \hbox{
% in } \R, \quad \phi_{\gamma}(0)=1, \quad \lim_{x\to +\infty} 
% \phi_{\gamma}(x)=0; \end{equation}
% \item there exists the limit $\mu (\gamma) := 
% -\lim_{x\to +\infty} \frac{1}{x}\ln \phi_\gamma (x)>0$.
% \end{enumerate}
% \end{prop}

% Obviously, Proposition \ref{prop-phigamma} is an immediate corollary of this result since 
% $\lambda_{1}(\L,\R) = c$ when $c$ does not depend on $x$. 

\begin{proof}[Proof of Proposition \ref{prop-phigamma}(i).] 
The conclusion holds even when $a$ and $c$ are not a.p.,
but just bounded so that $\lambda_1(\L,\R)$ is finite.
The proof is very close to that of Theorem 1.1 in \cite{Nolen}.
Take $\gamma>\lambda_1(\L,\R)$. Hence, $\gamma>\lambda_1 (\L,(0,R))$ for all
$R>0$, which implies that
the principal eigenvalue of the operator $\L-\gamma$ with Dirichlet boundary 
conditions in $(0,R)$ is negative. There exists then a unique positive
solution $\phi_\gamma^R$ of
\begin{equation}\label{eq-u_R}  \big( a(x)(\phi_\gamma^R)'\big)' + 
(c(x)-\gamma) \phi_\gamma^R  = 0 \ \hbox{ in }
(0,R), \quad \phi_\gamma^R(0)=1, \quad \phi_\gamma^R(R)=0. \end{equation} 
By the comparison principle, the family $(\phi_\gamma^R)_{R>0}$ is increasing in
$R$.
By Lemma \ref{lem:Nolen}, 
for all $\e>0$, there exists a constant $C=C (\e,\gamma)$ so that 
\begin{equation}\label{eq-expdecphi} 
\forall R>0,\ x\in (0,R),\quad
\phi_\gamma^R (x) \leq C e^{-\left(\sqrt{\gamma
-\lambda_1(\L,\R)} -\e\right) x}.\end{equation}
Hence, one can define $\phi(x) := \lim_{R\to +\infty} \phi_\gamma^R (x)$ for 
$x\geq0$. This 
limit belongs to $\mathcal{C}^2 ([0,\infty))$ 
and satisfies $\big( a(x)\phi'\big)'+ c(x)\phi = \gamma\phi$ over $(0,\infty)$ 
and $\phi (0)=1$.
Moreover, taking $0<\e < \sqrt{\gamma-\lambda_1 (\L, \R)}$, one gets from
(\ref{eq-expdecphi}) 
that $\lim_{x\to +\infty} \phi (x) = 0$. 

Next, we consider the unique $\mathcal{C}^2$ extension of $\phi$ to the full 
real line satisfying $\L\phi= \gamma \phi$. 
It is only left to check that this extension,
that we still denote $\phi$, is positive. Assume that it is not true and define
$x_0 := \sup \{ x\in\R, \ \phi (x) =0\}$.  
Define 
$$\varphi (x) := \left\{ \begin{array}{ccl} 
                               0 &\hbox{ if }& x<x_0,\\
\phi (x) &\hbox{ if }& x\geq  x_0.\\
                              \end{array}\right.$$
Then $\varphi \in H^1_0 (\R)$ since $\phi$ converges exponentially to $0$ as
$x\to +\infty$ and $\phi' \in L^{2} (x_{0},\infty)$ using equation $\big( a(x)\phi'\big)'+ c(x)\phi = \gamma\phi$. Taking $\varphi$ as a test-function in (\ref{def-RayleighR})
yields $\lambda_1 (\L,\R) \geq\gamma$, which is a contradiction.  

Lastly, the uniqueness follows   either from the characterization of the 
validity of the maximum principle in terms of the sign of $\l$ derived in 
\cite[Theorems 1.6 and 1.9]{BerRossipreprint}, or from Lemma \ref{lem:Nolen}. 
Indeed, for 
instance, applying  the latter with $x_0=0$ and $\phi$ equal to the difference 
of two solutions $\phi_1$, $\phi_2$ of (\ref{eq-prop-phigamma}), 
yields $\phi_1\leq\phi_2$ on $\R_+$. Then, exchanging $\phi_1$ and $\phi_2$ we 
eventually derive $\phi_1\equiv\phi_2$ on $\R_+$, and thus on the whole $\R$ by 
uniqueness of the Cauchy problem.
% satisfies 
% $$z(0)=0, \quad z \hbox{ is bounded}, \quad (\L-\gamma)z= 0 \ \hbox{ in } 
% (0,\infty).$$
% It follows from the same arguments as in Lemma 2.2 in \cite{Nolen} that  
% $z\equiv 0$. In other words, the solution $\phi$ of (\ref{eq-prop-phigamma}) is 
% unique.
% \footnote{Le fait d'utiliser non pas un resultat d'un autre papier mais plutot 
% les arguments de sa preuve pourrait deranger un 
% referee. Y a-t-il un enonce' precis qui implique le resultat?\\ 
% On pourrait 
% aussi le deduire par contraddiction (un peut comme dans le Lemma suivant): s'il 
% existe un interval $I$ ou' $z\neq0$ (disons $z>0$) et $z=0$ sur $\partial I$, 
% on aurait $\lambda'(\L,I)\geq\gamma$, et donc \cite[Theorem 
% 1.7]{BerRossipreprint} mene a' la contraddiction 
% $\lambda(\L,I)=\lambda'(\L,I)\geq\gamma$ .}
\end{proof}

\begin{lem}\label{lem-phiunbounded}
For all $\gamma > \lambda_1 (\L,\R)$, the function $\phi_\gamma$ is unbounded. 
\end{lem}

\begin{proof} Assume that $\phi_\gamma$ is bounded. Define 
$$\lambda_1' (\L,\R):= \sup \{ \lambda\in\R, \ \exists \varphi \in\mathcal{C}^2 
(\R)\cap L^\infty (\R),
 \ \phi>0 \hbox{ in } \R, \ \L \phi \geq \lambda \phi \hbox{ in } \R\}.$$
As $\phi_\gamma$ is bounded, one can take $\phi=\phi_\gamma$ in the above
formula, which gives $\lambda_1' (\L,\R) \geq \gamma$. On the other hand, 
it has been proved in \cite{BerestyckiHamelRossi} that, as $\L$ is self-adjoint, $\lambda_1'(\L,\R) = \lambda_1 (\L,\R)$. This contradicts $\gamma > \lambda_1 (\L,\R)$. 
\end{proof}

\begin{lem}\label{lem-ap}
%  Assume that $a$ and $c$ are a.p. Then, 
For all $\gamma >  
\lambda_1 (\L,\R)$,
the function $\phi_\gamma'/\phi_\gamma$ is a.p.
% 
% , where we denote $\phi_{\lambda_1(\L,\R)}$ 
% the limit $\lim_{\gamma\searrow \lambda_1 (\L,\R)} \phi_\gamma$ 
% defined by Lemma \ref{lem:cvcrit}. 
\end{lem}

\begin{proof} 
Take $\gamma > \lambda_1 (\L,\R)$.
Consider a sequence $(x_n)_n$ in $\R$. Then, up to subsequences, $(a(\cdot + x_n))_n$
and $(c(\cdot + x_n))_n$ converge uniformly to some
$a^*, c^*\in\mathcal{C} (\R)$. 
The operator $\mathcal{L}^*$ defined by
$\mathcal{L}^*\phi:=\big( a^*(x)\phi'\big)'+c^*(x)\phi$ is a limit operator associated with
$\mathcal{L}$. Hence, since $a$ and $c$ are a.p., Lemma 5.6 of
\cite{BerestyckiHamelRossi} yields $\lambda_1 (\L^*,\R)=\lambda_1
(\L,\R)<\gamma$.
We can therefore apply Proposition \ref{prop-phigamma}(i) to $\L^*$ and
infer that there exists a positive function $\phi^*\in\mathcal{C}^2 (\R)$ such
that
$$\L^*\phi^*  = \gamma \phi^* \ \hbox{ in } \R, \quad \phi^*(0)=1, \quad
\lim_{x\to +\infty} \phi^*(x)=0.$$ 
We prove the lemma by showing that $\phi_\gamma'(\cdot+x_n)/\phi_\gamma
(\cdot+x_n)$ converges up to subsequences to $(\phi^*)'/\phi^*$ uniformly in
$x\in\R$.
Assume by way of contradiction that this is not the case. There exists then a
sequence $(y_n)_n$ such that, up to extraction,
\begin{equation} \label{lemapcontrad} 
\lim_{n\to\infty}\left|\frac{\phi_\gamma
'(y_n+x_n)}{\phi_\gamma (y_n+x_n)} - \frac{(\phi^*)'(y_n)}{\phi^*
(y_n)}\right|>0.
\end{equation}
One can assume, always up to extraction, that as $n\to\infty$, $\big( a(\cdot + x_n+y_n) \big)_n$ and  $\big( c(\cdot + x_n+y_n) \big)_n$
converge to some $a^{**}$ and $c^{**}$ uniformly in $\R$. 
It is easy to check that 
$a^*(\cdot+y_n) \to a^{**}$ and $c^*(\cdot+y_n) \to c^{**} $ as $n\to\infty$ 
uniformly
in $x\in\R$. 

Next, set $\psi_n (x):= \phi_\gamma (x+x_n+y_n)/\phi_\gamma (x_n+y_n)$. Since 
$\phi_\gamma$ satisfies $\L\phi_\gamma=\gamma\phi_\gamma$, the 
Harnack inequality together with interior elliptic estimates imply that the 
sequence $(\psi_n)_n$ is bounded in $\mathcal{C}^{1,\alpha}(I)$ for any 
$\alpha\in(0,1)$ and any bounded interval $I$. It follows from the Ascoli
theorem 
that $(\psi_n)_{n\in\N}$ converges (up to extraction) in 
$\mathcal{C}^1_{loc}(\R)$ to some function $\psi_\infty$; expressing 
$\psi_n''$ from the equation $\L\psi_n=\gamma\psi_n$ we deduce that 
the convergence actually holds in $\mathcal{C}^2_{loc}(\R)$.
The function $\psi_\infty$ is positive and satisfies
\begin{equation}\label{eq:**}
\big(a^{**} (x)\psi_\infty'\big)' + c^{**}(x) \psi_\infty  = \gamma 
\psi_\infty \ \hbox{ in } \R, \quad \psi_\infty(0)=1.
\end{equation}
Furthermore, we know from Lemma \ref{lem:Nolen} that, for given $\e>0$, there 
is $C>0$ such that
$$\forall x>0,\quad
\phi_\gamma (x+x_n+y_n) \leq C\phi_\gamma (x_n+y_n)
e^{-\left(\sqrt{\gamma -\lambda_1(\L,\R)} -\e\right) x},$$
and thus $\psi_\infty (x) \leq  C e^{-\big(\sqrt{\gamma -\lambda_1(\L,\R)} 
-\e\big) x}$. This implies $\lim_{x\to +\infty} \psi_\infty (x) = 0$. 

Similarly, defining $\varphi_n (x) := \phi^*(x+y_n) /\phi^* (y_n)$ for all $n$
and $x\in\R$, one can prove that, up to extraction, the sequence $(\varphi_n)_n$
converges in $\mathcal{C}^2_{loc} (\R)$ to a solution $\varphi_\infty$ of 
\eqref{eq:**}. 
% $$\big(a^{**}(x)\varphi_\infty'\big)' + c^{**}(x) \varphi_\infty  = \gamma 
% \varphi_\infty \ \hbox{ in }
%  \R, \quad \varphi_\infty(0)=1, \quad \lim_{x\to +\infty} \varphi_\infty (x) = 
% 0.$$
Moreover $\lim_{x\to +\infty} \varphi_\infty (x) = 0$ again by Lemma 
\ref{lem:Nolen}. 
Proposition \ref{prop-phigamma} eventually yields that 
$\varphi_\infty \equiv \psi_\infty$. This is impossible because
$\varphi_\infty'(0)\neq\psi_\infty'(0)$ by (\ref{lemapcontrad}), which provides the final contradiction. 
\end{proof}

%\begin{lem}
%The function $x\mapsto \Big(\phi_\gamma'(x)/\phi_\gamma (x)\Big)^2$ is almost periodic uniformly with respect to $\gamma\in \big(\lambda_1(\L,\R),\Lambda\big]$
%for all $\Lambda>\lambda_1 (\L,\R)$. 
%\end{lem}
%
%\begin{proof}
%As 
%
%Assume that there exists $\e_0$ and sequences $(\gamma_k)_k$ and $(x_k)_k$ such that $\gamma_k \in \big(\lambda_1(\L,\R),\Lambda\big]$ for all $k$ and 
%$$\Big|\Big( \frac{\phi_{\gamma_k}' (x_k)}{\phi_{\gamma_k} (x_k)}\Big)^2 -\Big( \frac{\phi_{\gamma_l}' (x_l)}{\phi_{\gamma_l} (x_l)}\Big)^2\Big|\geq \e_0$$
%for all $k,l$ large enough. 
%It is easy to check that a contradiction with Lemma \ref{lem-ap} is reached if the sequence $(\gamma_k)_k$ does not converge to $\lambda_1 (\L,\R)$ as 
%$k\to +\infty$. 
% 
% 
%\end{proof}

\begin{proof}[Proof of Proposition \ref{prop-phigamma}(ii)]
We write
$$\frac1x \ln \phi_\gamma (x) =  \frac{1}{x}\int_0^x
\frac{\phi_\gamma'(y)}{\phi_\gamma (y)} dy.$$
Since the function $\phi_\gamma'/\phi_\gamma$ is a.p.~by Lemma 
\ref{lem-ap}, it is well known that the average
\begin{equation}\label{average}
-\mu(\gamma):=\lim_{x\to\pm\infty}\frac{1}{x}\int_z^{z+x}
\frac{\phi_\gamma'(y)}{\phi_\gamma (y)} dy
\end{equation}
exists uniformly in $z\in\R$ (see, e.g., 
\cite{Bes,Fink}). We show in the next 
lemma that $\mu (\gamma) \geq \sqrt{\gamma - \lambda_{1}(\L,\R)}>0$, 
which concludes the proof of the statement. 
\end{proof}

\begin{lem} \label{lem-convexitymu}
The function $\gamma \mapsto \mu (\gamma)$ defined on 
$(\lambda_1(\mathcal{L},\R),+\infty)$ is concave, nondecreasing and there 
exists $C>0$ such that, for $\gamma>\lambda_1(\mathcal{L},\R)$,
\begin{equation}\label{<mu<}
\sqrt{\gamma - \lambda_{1}(\L,\R)}\leq \mu (\gamma)\leq C\sqrt{\gamma},
\end{equation}
$$\mu (\gamma)>\underline\mu:= \lim_{\gamma\searrow \lambda_1 
(\L,\R)} \mu (\gamma).$$
\end{lem}

\begin{proof} Take $\lambda_1(\mathcal{L},\R)<\gamma_1<\gamma_2$ and $\theta\in
(0,1)$. Call $\gamma:= (1-\theta ) \gamma_1+\theta \gamma_2$ and 
$\psi (x) := \phi_{\gamma_1}^{1-\theta} (x) \phi_{\gamma_2}^\theta (x)$. A
straightforward computation yields 
\[\begin{array}{ll}
\big(a(x)\psi'\big)' &= -a(x) \theta (1-\theta) \psi \left(
\frac{\phi_{\gamma_1}'}{\phi_{\gamma_1}}-\frac{\phi_{\gamma_2}'}{\phi_{\gamma_2}
}\right)^2 + a(x)(1-\theta) \left(
\frac{\phi_{\gamma_2}}{\phi_{\gamma_1}}\right)^\theta 
\phi_{\gamma_1}''
+a(x)\theta \left( \frac{\phi_{\gamma_1}}{\phi_{\gamma_2}}\right)^{1-\theta}
\phi_{\gamma_2}'' \\ 

& + a'(x)(1-\theta) \left(
\frac{\phi_{\gamma_2}}{\phi_{\gamma_1}}\right)^\theta 
\phi_{\gamma_1}'
+a'(x)\theta \left( \frac{\phi_{\gamma_1}}{\phi_{\gamma_2}}\right)^{1-\theta}
\phi_{\gamma_2}'\\

&\leq (1-\theta)
\left(\frac{\phi_{\gamma_2}}{\phi_{\gamma_1}}\right)^\theta \left(
\gamma_1-c(x) \right)\phi_{\gamma_1}
+\theta \left( \frac{\phi_{\gamma_1}}{\phi_{\gamma_2}}\right)^{1-\theta}  \left(
\gamma_2-c(x) \right)\phi_{\gamma_2}\\ 
&=(\gamma-c(x))\psi.\\
\end{array}\]
The maximum principle, together with the boundary conditions at $0$ and $+\infty$, gives
$\phi_{\gamma} (x) \leq \psi (x) =\phi_{\gamma_1}^{1-\theta} (x)
\phi_{\gamma_2}^\theta (x)$ for all $x>0$. 
It follows that 
$$\forall x>0,\quad
-\frac1x\ln\phi_{\gamma} (x) \geq -\frac{(1-\theta)}{x} \ln
\phi_{\gamma_1}(x) -\frac{\theta}{x} \ln  \phi_{\gamma_2} (x),$$
whence, letting $x$ go to $+\infty$, we eventually get $\mu (\gamma) \geq
(1-\theta) \mu (\gamma_1) + \theta \mu(\gamma_2)$. 

Next,
$\phi_{\gamma_2}$ is a subsolution of the equation satisfied by 
$\phi_{\gamma_1}$ in $[0,\infty)$ and thus,  
% the Phragmen-Lindelof principle 
% yields 
applying Lemma \ref{lem:Nolen} 
to $\phi_{\gamma_2}-\phi_{\gamma_1}$ shows that
$\phi_{\gamma_2}\leq\phi_{\gamma_1}$ in $[0,\infty)$. Hence, we obtain the 
monotonicity of $\gamma \mapsto \mu (\gamma)$. 

Lemma \ref{lem:Nolen} yields that, for any $\e>0$, 
$$\mu(\gamma)=-\lim_{x\to +\infty} \frac{1}{x}\ln \phi_\gamma (x)\geq
-\lim_{x\to +\infty} \frac{C}{x}+\sqrt{\gamma - \lambda_{1}(\L,\R)}-\e
=\sqrt{\gamma - \lambda_{1}(\L,\R)}-\e.$$ 
Finally, assume by way of contradiction that there exists 
$\gamma>\lambda_{1}(\L,\R)$ 
such that $\mu (\gamma) = \underline{\mu}$. Then $\mu (\gamma) = 
\underline{\mu}$ for all $\gamma>\lambda_{1}(\L,\R)$ by monotonicity and 
concavity, which contradict $\mu (\gamma) \geq \sqrt{\gamma - 
\lambda_{1}(\L,\R)}$. 

Lastly, it is easy to check that Lemma 2.6 of 
\cite{Nolen} holds true in our framework, even if $a$ is 
heterogeneous, providing a constant $C$ such that 
$\mu (\gamma) \leq C \sqrt{\gamma}$.
\end{proof}

We conclude this section by the analysis of the limit of the exponent
$\mu(\gamma)$ as 
$\gamma\searrow \lambda_1$. 
% The proof of the next result involves similar arguments as in the 

\begin{prop}\label{prop:underlinemu}
Under Hypothesis \ref{hyp-crit} it holds that
$$\underline{\mu}:=\lim_{\gamma\searrow \lambda_1} \mu (\gamma)=0.$$
\end{prop}

\begin{proof}
Consider the analogue $\tilde\phi_\gamma$ of $\phi_\gamma$
but with imposed
decay at $-\infty$, namely,
\begin{equation*}
\L\tilde\phi_{\gamma}= \gamma \tilde\phi_{\gamma} \ \hbox{
in } \R, \quad \tilde\phi_{\gamma}(0)=1, \quad \lim_{x\to -\infty} 
\tilde\phi_{\gamma}(x)=0.
\end{equation*}
The function $\tilde\phi_\gamma(-x)$ fulfils the same properties as
$\phi_\gamma$. In particular the limit
$$\tilde\mu(\gamma) := 
\lim_{x\to\pm\infty} \frac{1}{x}\ln \tilde\phi_\gamma (x)$$
exists and it is positive. Consider the combination
$\varphi_\gamma:=\sqrt{\tilde\phi_\gamma\phi_\gamma}$.
The same computation as in the proof of 
Lemma \ref{lem-convexitymu} reveals that $\varphi_\gamma$ satisfies the
equation 
\footnote{ The function $\frac14a(x)q_\gamma^2$ is an 
{\em optimal Hardy weight} in the sense of \cite{Hardy}, which implies in 
particular that the operator $\L-\gamma+\frac14a(x)q_\gamma^2$ is critical. 
However, one cannot deduce from this the criticality of 
$\L-\lambda_1$ (and of its limit operators), but we need to impose it as an 
assumption.}
% There might be a link to be clarified, but we leave it to future 
% investigations.
\begin{equation}\label{vpg}
(\L-\gamma)\varphi_\gamma=-\frac14 a(x)q_\gamma^2\,\varphi_\gamma,\quad
\text{with }q_\gamma=\frac{\tilde\phi_\gamma'}{\tilde\phi_\gamma}
-\frac{\phi_\gamma'}{\phi_\gamma}.
\end{equation} 
We claim that $q_\gamma$ converges uniformly 
 to 0 in $\R$ as $\gamma\searrow \lambda_1$.
Suppose by contradiction that this is not the case.
Then there exist $\e>0$ and two sequences $(\gamma_n)_n$ and 
$(x_n)_n$ such that $\gamma_n\searrow \lambda_1$ and 
$|q_{\gamma_n}(x_n)|\geq\e$ for all $n\in\N$.
% Consider the translations
% \footnote{j'ai ajoute' des details qui me semblent necessaire car on pourrait
% avoir un probleme si $q_\gamma\to\infty$ as $agamma\to\lambda_1$} 
% $$\frac{\varphi_{\gamma_n}(x+x_n)}{\varphi_{\gamma_n}(x_n)}:=
% \sqrt{\frac{\tilde\phi_{\gamma_n}(x+x_n)\phi_{\gamma_n}(x+x_n)}
% {\tilde\phi_{\gamma_n}(x_n)\phi_{\gamma_n}(x_n)}}.$$
% Applying the Harnack inequality and then elliptic estimates to
% $\tilde\phi_\gamma$ and $\phi_\gamma$ we infer that the above translations are
% bounded in $C^2_{loc}(\R)$ independently of $n\in\N$.
Applying the Harnack inequality and then elliptic estimates to
$\tilde\phi_{\gamma_n}$ and $\phi_{\gamma_n}$ we see that the $(q_{\gamma_n})_n$
are equibounded and equicontinuous on $\R$.
Passing to the limit along (a subsequence of) the translations by $x_n$ in 
\eqref{vpg}, we find that the functions
$\varphi_\gamma(\cdot+x_n)/\varphi_\gamma(x_n)$ converge to 
a positive solution $\vp^*$ of 
$$(\L^*-\lambda_1)\varphi^*=-\frac14a^*(x)q^2\,\varphi^*,$$
where $\L^*$ is a limit operator associated with $\L$ and $a^*$, $q$ are the 
limits of $(a(\cdot+x_n))_n$ and $(q_{\gamma_n}(\cdot+x_n))_n$
respectively. Since $|q(0)|\geq\e$, we deduce that the operator
$\L^*-\lambda_1$ admits a positive supersolution which is not a 
solution. This contradicts the criticality of $\L^*-\lambda_1$ due 
to \cite[Theorem~3.9 at p.~152]{Pinsky}. We have therefore shown that
$q_\gamma\to0$ 
uniformly in $\R$ as~$\gamma\searrow \lambda_1$. Finally, because
$$\tilde\mu(\gamma)+\mu(\gamma)=\lim_{x\to+\infty}\frac1x\int_0^x 
q_\gamma\leq\|q_\gamma\|_\infty,$$
and $\tilde\mu(\gamma)$, $\mu(\gamma)$ are positive, we infer that
both $\tilde\mu(\gamma)$ and $\mu(\gamma)$ tend to $0$ as $\gamma\searrow 
\lambda_1$.
% 
% La notion de criticality n'est pas trop aimee dans ``notre ecole'' et 
% est basee 
% sur l'existence de la fonction de Green.
% Mais en dimension 1 est equivalent a' dire que $\lambda_1$ est simple (parmi 
% les 
% fonction propres positives).
% Il y a plusieurs conditions suffisantes connues pour cela. 
% Il faudrait fouiller 
% dans le bouquin de Pinsky,
% ou dans les papiers de Pinchover par example, pour voir si toute operateur 
% a.p. 
% auto-adjoint est critical (a' tester sur l'operateur de ton contre-example
% pour la non-existence de fonction propre a.p.).
\end{proof}
%\footnote{On exploite que tres peu les proprietes des $q_\lambda$.
%Par exemple on sait que $q_\gamma>0$ partout par l'unicite' du probleme de 
%Cauchy (je pense meme $\inf q_\gamma>0$).
%
%Je suis d'accord sur l'inf positif. C'est peut etre la qu'on peut ameliorer par exemple vers la conjecture de Pinchover plus hat. Mais de toutes facons on sait que c'est pas vrai en general que $\underline{\mu}>0$. Et la on ne sait pas comment relaxer l'hypothese 2 (avec des exemples ou ca serait utile). Donc je serais pour en rester la. 
%}

\begin{rmk}\label{rmk:hyp}
As Hypothesis \ref{hyp} implies Hypothesis \ref{hyp-crit} by  
Proposition \ref{prop:hypimply}, it follows from Proposition
\ref{prop:underlinemu} that $\underline{\mu}=0$ when Hypothesis \ref{hyp}
is fulfilled. 
Owing to Lemma \ref{lem-convexitymu}, this implies in particular 
that $w^*$ is attained. Therefore, Theorem~\ref{thm:gtf} is a consequence of 
 Theorem \ref{thm-apgen}.
\end{rmk}

\begin{rmk}\label{rmk:c}
As explained in Section \ref{sec:hyp}, Hypotheses \ref{hyp} and \ref{hyp-crit} 
are fulfilled if $c$ is constant. However, in such case, the limit 
of $\mu (\gamma)$ can be derived directly without using the theory of critical 
operators.
Assume indeed that $c$ is a positive constant. 
On one hand, taking $\phi\equiv1$ in \eqref{charI} yields $\lambda_1\leq c$, on
the other, we know from Proposition \ref{lem-caraclambda1} that
$\lambda_1\geq c$. Thus, $\lambda_1=c$.
We use now the same type of argument as Zlatos in the proof of 
\cite[Theorem 1.1]{ZlatosKPP}, even if our aim is different. 
 Let $\delta>0$, take $\gamma >c$ and define $u(x):=a(x)  
\phi_{\gamma}'(x)/\phi_{\gamma}(x)$. 
This function is a.p.~- being the product of two a.p.~functions, see, e.g., 
\cite[Theorem 1.13]{Bes} - and satisfies 
$$u' +u^{2}/a(x)= \gamma - c \quad \hbox{ in } \R.$$
Moreover, as $\big( a(x) \phi_{\gamma}'\big)' = (\gamma - c)\phi_{\gamma}>0$ in
$\R$, the function $\phi_\gamma$ does not admit any local maximum and, as
$\phi_{\gamma}(+\infty)=0$, it is thus nonincreasing. Hence $u$ is nonpositive. 
Now, it is easy to check that 
$$u^{2}(x)\leq (\gamma - c)\sup_{\R}a  \quad \hbox{ for all } x\in \R,$$
because if the above inequality failed for some $x_0\in\R$, $u$ would be 
decreasing in $(-\infty,x_0)$, contradicting the almost periodicity.
It follows that 
$$ |\phi_{\gamma}'(x)/\phi_{\gamma}(x)| \leq 
\frac{\sqrt{\gamma-c}\,\sup_{\R}\sqrt{a}}{\inf_{\R}a} \quad 
\hbox{ for all } x\in \R.$$
Recalling that $\mu(\gamma)$ is equal to the average of 
$\phi_{\gamma}'/\phi_{\gamma}$, we derive
$$|\mu (\gamma )| \leq \frac{\sqrt{\gamma-c}\,\sup_{\R}\sqrt{a}}{\inf_{\R}a} 
\quad 
 \hbox{ for all } \gamma >c,$$
from which \eqref{mu=0} follows because $\lambda_1=c$.
\end{rmk}

%%%%%%%%%%%%%%%%%%%%%%%%%%%%%%%%%%%%%%%%%%%%%%%%%%%%%%%%%%%%%%%%%%

\section{Construction of the fronts} \label{sec:construction}

%This section is dedicated to the proof of the following result.
%
%\begin{prop}\label{pro:gtf} 
%Assume that 
%$$w^*:=\inf_{\gamma>\lambda_1 (\L,\R)} \displaystyle\frac{\gamma}{\mu 
%(\gamma)}<\frac{\lambda_1 (\L,\R)}{\underline{\mu}}=:\underline{w},$$
%where $\mu (\gamma)$ is defined in Proposition \ref{prop-phigamma} and 
%$\underline{\mu}:= \lim_{\gamma\searrow \lambda_1 (\L,\R)} \mu (\gamma)$. 
%Then $w^*$ is attained and for all $w>w^*$, there exists a (weak)
%solution $u\in\mathcal{C}^{0}(\R\times \R)$ of (\ref{eqprinc}) which
%can be written as $u(t,x) = U (\int_0^x\sigma -t,x)$, where 
%$\sigma\in\mathcal{C} (\R)$ is a.p.~and has average $\langle 
%\sigma 
%\rangle = 1/w$, $U=U(z,x) \in \mathcal{C}^{0} (\R\times \R)$
%is a.p.~in $x$ uniformly in $z\in\R$, decreasing in $z$ and satisfies
%\begin{equation} \label{eq-cvU} \lim_{z\to -\infty} U (z,x) = 1 \hbox{\quad 
%and\quad } \lim_{z\to +\infty} U (z,x) =0 \quad\hbox{ uniformly in } 
%x\in\R.\end{equation}
%\end{prop}
%
%This will immediately entail property (i) of Theorems \ref{thm:gtf} and 
%\ref{thm-apgen}. Indeed, if Hypothesis \ref{hyp} is satisfied, then $\underline{\mu}=0$ in this case by Proposition  \ref{prop:underlinemu} and thus  $\underline{w}= +\infty$.

\subsection{A preliminary result on a.p.~linear operators}

Consider an arbitrary elliptic operator defined for all $\phi \in\mathcal{C}^2 (\R)$ by
\begin{equation}
 \mathcal{M} \phi := a(x)\phi'' + b(x) \phi' + c(x) \phi,
\end{equation}
with $a$, $b$, $c$ a.p.~in $x$ and $\inf_\R
a>0$.

\begin{lem}\label{lem-zetatheta}
 Assume that there exist $\delta'\in\R$ and a positive function $\zeta
\in\mathcal{C}^2 (\R)$ so that 
$$-\mathcal{M} \zeta \geq \delta' \zeta\quad\text{in }\R,\qquad
\lim_{|x|\to +\infty} \frac{1}{x}\ln \zeta (x) = 0.$$ Then for all $\delta
<\delta'$, there exists
$\theta \in\mathcal{C}^2 (\R) \cap L^\infty (\R)$ satisfying 
$$-\mathcal{M} \theta\geq \delta \theta\quad\text{in }\R,\qquad\inf_\R
\theta>0.$$
\end{lem}

\begin{proof} Our proof makes use of the following two generalized principal
eigenvalues introduced in \cite{BN1}: 
\begin{equation}\label{eq:lambda1}
\underline{\lambda_1}:=\sup\{\lambda\ |\
\exists\zeta\in \mathcal{C}^2 (\R), \ \zeta>0, \ \lim_{|x|\to +\infty} \frac{1}{x} \ln \zeta (x)=0 \hbox{ such that } \mathcal{M}\zeta\geq\lambda\zeta \hbox{ in } \R \},
\end{equation}
\begin{equation}\label{eq:lambda12}
\overline{\lambda_1}:=\inf\{\lambda\ |\
\exists\zeta\in \mathcal{C}^2 (\R), \ \zeta>0, \ \lim_{|x|\to +\infty} \frac{1}{x} \ln \zeta (x)=0 \hbox{ such that } \mathcal{M}\zeta\leq\lambda\zeta \hbox{ in } \R \}.
\end{equation}
It has been proved in \cite{BN1} that $\underline{\lambda_1} \leq
\overline{\lambda_1}$. The hypothesis of the lemma reads
$\overline{\lambda_1} \leq -\delta'$. 

We will now use the same type of arguments as in \cite{BN1} in order to show 
that $\underline{\lambda_1} = \overline{\lambda_1}$, from which we will 
construct the desired supersolution $\theta$. 
As the result of this lemma plays a central role
in the proof of our main result, we briefly sketch the argument here for the
sake of completeness. 
Consider the family of equations defined, for $\e>0$, by:
\begin{equation}\label{eq-approxap}
a(x) u_\e''+a(x) (u_\e')^2 +b(x) u_\e' +c(x)=\e u_\e \quad\hbox{in } \R.
\end{equation}
For all $\e>0$, as $c$ is uniformly bounded, there exists some large $M_\e$ such that $-M_\e$ is a subsolution and 
$M_\e$ is a supersolution of (\ref{eq-approxap}). It follows from Perron's
method that there exists a (unique) bounded solution
$u_\e\in\mathcal{C}^{2}(\R)$ of 
equation (\ref{eq-approxap}).
It has been proved by Lions and Souganidis in \cite[Lemma
3.3]{LionsSouganidisap} that $\lambda = \lim_{\e\to 0} \e u_\e (x)$ exists
uniformly in $x\in\R$.

The function $\zeta:= e^{u_\e}$ satisfies $\mathcal{M}\zeta=\e u_\e\zeta$.
Taking it as a test function in the definitions of
$\underline{\lambda_1}$ and $\overline{\lambda_1}$ and letting $\e$ go to $0$
yields
$ \underline{\lambda_1} \geq \lambda\geq \overline{\lambda_1} $ and thus 
$ \underline{\lambda_1} =\lambda=\overline{\lambda_1}$. In particular, $\lambda \leq -\delta'.$
Take $\delta <\delta'$ and choose $\e>0$ small enough so that $|\e u_\e
(x) -\lambda| <
\delta'-\delta$ for all $x\in\R$. Then, the function $\theta := \zeta=e^{u_\e}$
satisfies 
$$
\mathcal{M}\theta=\e u_\e\theta \leq (\lambda+\delta'-\delta)\theta
\leq-\delta \theta\quad \hbox{ in } \R.
$$
As $u_\e$ is bounded, $\theta$ is also bounded and satisfies $\inf_\R\theta>0$.
\end{proof} 

%-----------------------------------------------------------------------------------------------------------

\subsection{Construction of sub and supersolutions}

In the sequel, for $\gamma>\lambda_1$, we will let $\phi_\gamma$ denote the
function given by Proposition \ref{prop-phigamma}. We further set
$$\sigma_\gamma:=-\frac{\phi_\gamma'}{\phi_\gamma}.$$

\begin{lem} \label{lem:decmu}
 Assume that $w^*,\underline{w}$ defined in Theorem \ref{thm-apgen}
satisfy 
$w^*<\underline{w}$.
 Then $w^*$ is a minimum. Moreover, for all $w\in (w^*,\underline{w})$, 
there exists $\gamma>\lambda_1$ such that $w=\gamma/\mu (\gamma)$ and 
$w>\gamma'/\mu (\gamma')$ for $\gamma'-\gamma>0$ small enough.
\end{lem}

\begin{proof}
Recall that $\lambda_{1}\geq \inf c >0$ by Proposition \ref{lem-caraclambda1}. 
On one hand, by hypothesis,
 $$\lim_{\gamma \to \lambda_1} \frac\gamma{\mu 
(\gamma)}=\frac{\lambda_1}{\underline{\mu}}
=\underline{w}>w^*=\inf_{\gamma>\lambda_1}\frac{\gamma}{\mu(\gamma)}.$$
On the other, the upper bound for $\mu$ in \eqref{<mu<} yields $\gamma/\mu
(\gamma)\to+\infty$ as $\gamma \to 
+\infty$ . Hence, the function $\gamma \mapsto \gamma/\mu(\gamma)$,  
which is continuous because 
$\gamma\mapsto \mu (\gamma)$ is concave by Lemma~\ref{lem-convexitymu}, admits a
minimum on
$(\lambda_1,+\infty)$. Let $\gamma^*$ be a minimizing point.
 
 Take $w\in (w^*,\underline{w})$ and define $f(\gamma):= w\mu (\gamma) 
-\gamma$ for $\gamma\in(\lambda_1,+\infty)$.  There holds 
$$\lim_{\gamma\to+\infty}f(\gamma)= -\infty,\qquad
\lim_{\gamma\searrow\lambda_1}
f(\gamma)=w \underline{\mu}-\lambda_1<0.$$ 
 Moreover, $f$ is concave and $f(\gamma^*)>0$.
%  where $\gamma^*$ minimizes $w^*=\min_{\gamma'>\lambda_1 
% (\L,\R)}\frac{\mu(\gamma')}{\gamma'}$. 
It follows that $f$ is strictly increasing for $\gamma$ less than its first 
maximal point $\gamma_M$, and that $f(\gamma_M)>0$. The unique zero of $f$ in 
$(\lambda_1 ,\gamma_M)$ provides us with the desired $\gamma$.
% there exists $\gamma \in \big(\lambda_1 (\L,\R),\gamma^*)$ such 
% that $f(\gamma)=0$ and $f$ is increasing in a neighborhood of $\gamma$ by 
% concavity. 
%  The conclusion follows.  
\end{proof}

Throughout this section, 
we take $w\in (w^*,\underline{w})$ and we let $\gamma>\lambda_1 $ be as in 
Lemma~\ref{lem:decmu}. \\
For a given a.p.~function $\sigma \in\mathcal{C}^1 (\R)$, we define the operator
$$
\L_\sigma\phi:= e^{\int_0^x \sigma} \L \left(e^{-\int_0^x
\sigma} \phi\right)
= \big( a(x)\phi'\big)' -2 a(x)\sigma\phi' +
\left(a(x)\sigma^2-(a(x)\sigma)'+c(x)\right)\phi.
$$

\begin{prop} \label{prop-theta}
There exist $\delta>0$, $\e\in(0,1)$ and a function $\theta \in\mathcal{C}^2
(\R)
\cap L^\infty (\R)$ satisfying
\begin{equation*} 
\inf_\R \theta>0,\qquad
-\L_{(1+\e)\sigma_\gamma} \theta\geq
\left(\delta-(1+\e)\gamma\right) \theta\quad\text{in }\;\R.
\end{equation*}
\end{prop}

\begin{proof} Due to our choice of $\gamma$, there exists $\e\in(0,1)$ so that 
$$\frac{\gamma}{\mu (\gamma)} > \frac{(1+\e)\gamma}{\mu \left( (1+\e) \gamma
\right)}.$$
Define 
$$F(\kappa):= \displaystyle\frac{1}{\mu (\gamma)} - \frac{1+\e}{\mu (\gamma +\kappa)}.$$ 
Then 
$F(\gamma \e ) =1/\mu (\gamma) - (1+\e)/\mu \left( (1+\e) \gamma \right)>0$
and $F(0)= -\e /\mu (\gamma)<0$. 
As $\mu$ is concave, it is continuous and thus $F$ is continuous. Hence, there
exists $\kappa\in (0,\e \gamma)$ so that $F(\kappa)=0$. 

Consider now  the function
$$\zeta (x):= \frac{\phi_{\gamma+\kappa}(x)}{\phi_\gamma^{1+\e}(x)} \hbox{ for all } x\in\R.$$
Keeping in mind that $e^{-(1+\e)\int_0^x \sigma_\gamma}=\phi_\gamma^{1+\e}$, 
we find that the positive function $\zeta$ satisfies, in $\R$,
$$\L_{(1+\e) \sigma_\gamma} \zeta = e^{(1+\e)\int_0^x \sigma_\gamma} \L \left(
e^{-(1+\e)\int_0^x \sigma_\gamma} \zeta\right)
=\frac{1}{\phi_\gamma^{1+\e}} \L \left(\zeta
\phi_\gamma^{1+\e}\right)=\frac{1}{\phi_\gamma^{1+\e}} \L \left(
\phi_{\gamma+\kappa}\right)= (\gamma+\kappa) \zeta.$$
It follows that 
$$-\big(\L_{(1+\e) \sigma_{\gamma}}-(1+\e)\gamma\big) \zeta = (\e\gamma-\kappa) \zeta
\quad \hbox{ in } \R.$$
Moreover, 
$$-\frac{1}{x} \ln \zeta (x) = -\frac{1}{x} \ln \phi_{\gamma +\kappa}(x)
+\frac{1+\e}{x} \ln \phi_\gamma (x),$$
and we know from Proposition \ref{prop-phigamma}(ii) that,  as $x\to +\infty$, 
the right-hand side tends to $\mu(\gamma+\kappa)-(1+\e) \mu (\gamma)$, which is 
equal to $0$ by the definition of $\kappa$. Notice that we obtain the same 
limit when $x\to -\infty$ thanks to the fact that $\mu(\gamma)$ is the {\em 
uniform} average of the a.p.~function $-\phi_\gamma'/\phi_\gamma$, that is, 
\eqref{average} holds uniformly with respect to $z\in\R$,  and likewise
$\mu(\gamma+\kappa)$ is the uniform average of 
$-\phi_{\gamma+\kappa}'/\phi_{\gamma+\kappa}$.
Hence, the hypotheses of Lemma \ref{lem-zetatheta} are fulfilled by the operator
$\mathcal{M}=\L_{(1+\e)\sigma_{\gamma}}-(1+\e)\gamma$ with
$\delta'=\e\gamma-\kappa>0$, and the statement of the proposition follows, up 
to showing that $\mathcal{M}$ has a.p.~coefficients.
Since $a$ and 
$\sigma_\gamma$ are a.p.~by hypothesis and Lemma \ref{lem-ap}, the only 
nontrivial check concerns the term $(a\sigma_{\gamma})'$. We know from
\cite[Theorem 
1.16]{Fink} that, in order to prove that $(a\sigma_{\gamma})'$ is a.p., it is
sufficient 
to show that $a\sigma_{\gamma}$ is uniformly continuous, which is readily
achieved  
applying the Harnack inequality and a priori estimates to the function 
$\phi_\gamma$.
\end{proof}

\bigskip

Define for all $(t,x) \in \R\times \R$:
$$
 \overline{u} (t,x):=  \min \big\{ 1, \phi_\gamma (x) e^{\gamma t} \big\}, 
$$
$$
 \underline{u} (t,x):=  \max \big\{ 0, \phi_\gamma (x) e^{\gamma t} - 
A\,\theta(x)\phi_{\gamma}^{1+\e}(x) e^{(1+\e)\gamma  t} \big\}, 
$$
where $\e$ and $\theta$ are given by Proposition \ref{prop-theta}
and $A$ is a positive constant that will be chosen later. 

\begin{prop} \label{prop-existenceu}
 There exists a (weak) solution $u\in\mathcal{C}(\R\times \R)$ of
(\ref{eqprinc}) satisfying $\underline{u}\leq u \leq \overline{u}$ in $\R\times\R$. 
%Moreover, $u$ satisfies
%\begin{equation}\label{infleft}
%\inf_{t\geq\frac1\gamma\int_0^x\sigma_\gamma}u(t,x)>0. 
%\end{equation}
Moreover, $u=u(t,x)$ is increasing in $t$. 
\end{prop}

\begin{proof} 
Direct computation reveals that $\phi_\gamma (x) e^{\gamma t}$ is a 
supersolution on the whole $\R\times \R$ of~(\ref{eqprinc}), hence 
$\overline{u}$ is a generalized supersolution of the same equation. 
Take~$(t,x) \in \R \times \R$ so that $ \underline{u} (t,x) >0$ and set for 
short $\zeta:=\phi_\gamma(x)e^{\gamma t}$. One has:
\[\label{eq-underlineu} \begin{array}{rcl}
\underline{u}_t - \big(a(x)\underline{u}_x\big)_x - c(x) \underline{u} &=& 
-A (1+\epsilon) \gamma\theta \phi_{\gamma}^{1+\epsilon}e^{
(1+\epsilon)\gamma t}+ Ae^{(1+\epsilon)\gamma t} 
\phi_{\gamma}^{1+\epsilon}\mathcal{L}_{(1+\epsilon)\sigma_{\gamma}}\theta\\
&&\\
&=& A\zeta^{1+\e}[\L_{(1+\e)\sigma_\gamma}\theta-(1+\e)\gamma \theta]\\
&&\\
&\leq&-A\delta \theta \zeta^{1+\e}.\\
                 \end{array}
\]
Therefore, as $0$ obviously solves (\ref{eqprinc}), for $\underline u$ to be a subsolution it is 
sufficient to choose $A$ so large that, for all $(t,x)$ such that 
$\underline{u} (t,x)>0$,  one has 
$$
A\delta \theta \zeta^{1+\e}\geq  c \zeta^2.$$
Observe that $\underline{u} (t,x)>0$ if and only if $A\theta 
(t,x)\zeta^{\epsilon}(t,x)<1$, i.e.,
$\zeta^{\epsilon-1}(t,x)>(A\theta(t,x))^{1/\e-1}$, whence the goal is achieved 
for
$$A\geq \frac{\sup_\R c^{\e}}{\delta^{\e} \inf_{\R}\theta}.$$
The above observation also shows that $\underline{u}<(A\theta)^{-1/\e}$, and 
thus $A$ can be chosen in such a way that, in addition, $\underline{u}<1$,
whence $\underline{u} \leq \overline{u}$. 
A standard argument then provides us with a solution $\underline{u}\leq u \leq \overline{u}$.
Let us recall such argument and show that $u$ inherits from $\overline{u}$ the monotonicity
in $t$.

Define the sequence of function $(u_n)_n$ in the following way:
$u_n$ is the solution to (\ref{eqprinc}) for $t>-n$ with initial condition
$u_n(-n,x)=\overline u(-n,x)$. By the comparison principle, the~$u_n$ satisfy
$$\forall t>-n,\ x\in\R,\quad
\underline u(t,x)\leq u_n(t,x)\leq\overline u(t,x).$$ 
Thus, for $m,n\in\N$ with $m<n$ and for any $0<h<1$, using the monotonicity of $\overline u$ we get
$$u_m(-m,x)=\overline u(-m,x)\geq \overline u(-m-h,x) \geq u_n(-m-h,x).$$
Observe that $u_n(\cdot-h,\cdot)$ is also a solution of (\ref{eqprinc}), whence the comparison principle
yields 
\begin{equation}\label{incseq}
\forall m<n,\ 0<h<1,\ t>-m,\ x\in\R,\quad
u_m(t,x)\geq u_n(t-h,x).
\end{equation}
By interior parabolic estimates, $(u_n)_n$ converges locally uniformly (up to subsequences) 
to an entire solution $\underline{u}\leq u \leq \overline{u}$ of (\ref{eqprinc}).
Then, passing to the limit as $m,n\to\infty$ (along a subsequence) in \eqref{incseq} we 
eventually infer that $u(t,x)\geq u(t-h,x)$ for all $(t,x)\in\R\times\R$ and $0<h<1$.
This means that $u(t,x)$ is nondecreasing in $t$.
If the monotonicity were not strict, 
the parabolic strong maximum principle would imply
that $u$ is constant in time, contradicting the fact that $\underline{u}\leq u \leq \overline{u}$.
\end{proof}

%----------------------------------------------------------------------------

\subsection{Asymptotics of the profile $U$}

We define 
\begin{equation} \label{def-U} U (z,x):= u\left( \frac{1}{\gamma}\int_0^x
\sigma_\gamma-z,x \right) \hbox{ for all } (z,x) \in\R\times \R,\end{equation}
where $u$ is the function given by Proposition \ref{prop-existenceu} and, we recall, 
$\sigma_\gamma:=-\phi_\gamma'/\phi_\gamma$.

\begin{lem}\label{lem:cvU}
 The function $U=U(z,x)$ is decreasing in $z$ and satisfies 
 \begin{equation} \label{eq-cvU}
 \lim_{z\to +\infty} U(z,x)=0, \quad  \lim_{z\to -\infty} U(z,x)=1, \quad \hbox{ uniformly in } x\in \R. 
 \end{equation}
\end{lem}

\begin{proof} 
The monotonicity in $z$ of $U$ follows from the monotonicity in $t$ of $u$. 
One has 
$$U (z,x) \leq \overline{u} \left( \frac{1}{\gamma}\int_0^x \sigma_\gamma-z,x
\right)
\leq \phi_\gamma (x) \exp \left( \int_0^x \sigma_\gamma-\gamma
z\right)=e^{-\gamma z}.$$ 
Similarly, for $(z,x) \in\R\times \R$,
$$\begin{array}{rcl} U (z,x) &\geq& \underline{u} \left( \frac{1}{\gamma}\int_0^x \sigma_\gamma-z,x
\right) \geq e^{-\gamma z} - A\theta (x) \phi_{\gamma}^{1+\e}(x)
e^{(1+\e)\int_0^x \sigma_\gamma - (1+\e)\gamma z} \\
&\geq & e^{-\gamma z} - A\left(\sup_\R\theta\right) e^{- (1+\e)\gamma z}.\\ 
\end{array}$$ 
Namely, $U$ satisfies
\begin{equation}\label{eq:exp}
\forall (z,x)\in\R\times\R,\quad
e^{-\gamma z}(1-Me^{-\e\gamma z})\leq U(z,x)\leq e^{-\gamma z},
\end{equation}
for some positive constant $M$. 
From the second inequality we deduce that $U(z,x) \to 0$ as $z\to +\infty$
uniformly with respect to $x\in\R$. From the first one, we see that
$\inf_{x\in\R} U (z,x)>0$ for $z$ large enough, and therefore, because of the monotonicity in $z$,
\begin{equation}\label{eq:inf}
\forall z\in\R,\quad \inf _{(-\infty,z]\times\R}\, U>0.
\end{equation}
%
%
%. But, as $U$ is decreasing in 
%$z$, this implies that 
%
%
%
%Moreover, \eqref{infleft} translates into
%\begin{equation}\label{U>0}
%\inf_{(-\infty,0]\times\R}\,U>0.
%\end{equation}
%From the second inequality in \eqref{eq:exp} we deduce that $U \to 0$ as $z\to
%+\infty$ uniformly with respect to $x\in\R$. Define the quantity
In particular, the following quantity is positive:
$$
\vartheta:=\lim_{z\to-\infty}\left(\inf_{x\in\R}U(z,x)\right).
$$
To conclude the proof we need to show that $\vartheta=1$.
Let $(x_n)_n$ in $\R$ be such that $U(-n,x_n)\to\vartheta$ as $n\to\infty$.
Consider the family of functions $(p^n)_{n}$ defined by
$$p^n(t,x):=u\left(
\frac{1}{\gamma}\int_0^{x_n}\sigma_\gamma+n-t,x_n+x\right)=
U\left(\frac{1}{\gamma}\int_{x_n}^{x_n+x}\sigma_\gamma-n+t,x_n+x\right).$$
These functions satisfy $p^n(0,0)=U(-n,x_n)\to\vartheta$ as $n\to\infty$ and,
for $(t,x)\in\R\times\R$, $\liminf_{n\to\infty}p^n(t,x)\geq\vartheta$.
Moreover, by interior parabolic estimates, they converge, up to subsequences,
to a function $p^\infty$ satisfying $\partial_t
p^\infty-\big( a^{*}(x)p^\infty_{x}\big)_{x} = c^*(x)p^\infty(1-p^\infty)$ in $\R\times\R$, where
$a^{*}$ and $c^*$ are the uniform limits of $\big(a(\cdot+x_n)\big)_n$ and $\big(c(\cdot+x_n)\big)_n$ (up to extraction). Evaluating this equation at
the minimum point $(0,0)$ of $p^\infty$ yields
$c^*(0)\vartheta(1-\vartheta)\leq0$. Since $c^*(0)\geq\inf_\R c>0$ and
$0<\vartheta\leq1$, we eventually get $\vartheta=1$.
\end{proof}

%----------------------------------------------------------------------------

\subsection{Almost periodicity of the profile $U$}
%
%\begin{equation} \label{hyp-vois1}
%\forall x\in\R,\quad u\mapsto f(x,u)/u \hbox{ is nonincreasing in }
%[1-\delta,1].
%\end{equation}

The function $U$ defined by (\ref{def-U}) solves 
\begin{equation}\label{deg}
-\sigma^2(x)a(x)\partial_{zz}U-\partial_x\big( a(x)\partial_x
U\big)+2\sigma(x)a(x)\partial_{zx}U-
((a\sigma)'(x)+1)\partial_z U=c(x)U(1-U),
\end{equation}
for $z\in \R$, $x\in \R$, 
with $\sigma:=\sigma_\gamma/\gamma$. 
We say that a function $V$ is a sub (resp.~super) solution of (\ref{deg}) in a
domain
$\mathcal{O}\subset\R^2$ if $v(t,x):=V(\int_0^x\sigma-t,x)$
is a classical sub (resp.~super) solution of (\ref{eqprinc}) for
$(\int_0^x\sigma-t,x)\in\mathcal{O}$.

The following result is an easy consequence of the weak Harnack inequality.

\begin{lem}\label{lem:I}
Let $I$ be an open interval, $\sigma\in W^{1,\infty}(\R)$ and $U^1,\ U^2$ be
respectively a sub and a supersolution of (\ref{deg})
which are uniformly continuous and
satisfy $0\leq U^1\leq U^2\leq1$ in $I\times\R$. If 
$$\exists z\in I,\quad\inf_{x\in\R}(U^2-U^1)(z,x)=0,$$ then 
$$\forall z'\in I,\ z'>z,\quad\inf_{x\in\R}(U^2-U^1)(z',x)=0.$$
\end{lem}

\begin{proof}
Assume that there exists $z\in I$ such that $\inf_{x\in\R}(U^2-U^1)(z,x)=0$ and
take $z'\in I$, $z'>z$. 
Let $(x_n)_n$ be such that $(U^2-U^1)(z,x_n)\to 0$ as
$n\to\infty$. Take $h>0$ small enough so that $z-h,z'+2h\in I$ and take
$0<\rho<h/\|\sigma\|_{L^\infty(\R)}$. Then, for $i\in\{1,2\}$ and $n\in\N$,
the function
$$u^i_n(t,x):=U^i\left(\int_0^x\sigma(x_n+s)ds-t,x_n+x\right)$$
is well defined in $[-z'-h,-z]\times(-\rho,\rho)$.
We define there the function $w_n:=u^2_n-u^1_n$, which is nonnegative 
and satisfies
$$\partial_t w_n-\partial_x \big( a(x+x_{n})\partial_x w_n\big)\geq
c(x_n+x)(1-u^2_n-u^1_n)w_n,\quad -z'-h
\leq t\leq-z,\ |x|<\rho.$$
Therefore, taking $z<z_1<z_2<z'$, the parabolic weak Harnack inequality (see
e.g. Theorem 7.37 in \cite{Lie}) provides two constants $p,C>0$ such that
$$\forall n\in\N,\quad
\|w_n\|_{L^p((-z',-z_2)\times(-\frac\rho2,\frac\rho2))}
\leq C\inf_{(-z_1,-z)\times(-\frac\rho2,\frac\rho2)}w_n
\leq C w_n(-z,0).$$
Whence, since $\lim_{n\to\infty}w_n(-z,0)=0$, the $(w_n)_n$ converge to $0$ in
$L^p((-z',-z_2)\times(-\frac\rho2,\frac\rho2))$. 
One readily deduces from the equicontinuity of the $(w_n)_n$ that the above
$L^p$ convergence to $0$ can hold only if 
$w_n(-z',0)\to0$ as $n\to\infty$. We eventually infer that 
$\inf_{x\in\R}(U^2-U^1)(z',x)=0$.
% Because the $(w_n)_n$ are equicontinuous, the $L^p$ convergence implies 
% that they are equibounded.
% We can therefore apply the Arzela-Ascoli theorem and infer that $(w_n)_n$ 
% converges (up to subsequences) uniformly in 
% $[-z',-z_2]\times[-\frac\rho2,\frac\rho2]$ to~some function, which turns out
% to 
% be identically equal to 0 due to the $L^p$ convergence.
% Then, in particular, $\lim_{n\to\infty}w_n(-z',0)=0$, which 
% This concludes the proof.
\end{proof}

The proof of the almost periodicity of $U$ is based on a
sliding method. 

\begin{prop} \label{prop:apU}
The profile $U=U(z,x)$ is a.p.~in $x$
uniformly with respect to $z\in \R$. 
\end{prop}

\begin{proof} 

Consider a sequence $(x_n)_n$ in $\R$. By the almost periodicity of $a$, $c$ and
$\sigma$, we can assume that, up to extraction of a subsequence, $(a(\cdot+x_n))_n$,
$(c(\cdot+x_n))_n$ and $(\sigma(\cdot+x_n))_n$ converge uniformly in $\R$.
We claim that $(U(\cdot,\cdot+x_n))_n$ converges uniformly in
$\R\times\R$.
Assume by contradiction that it is not the case. There exist then two
subsequences $(x^1_n)_n$,
$(x^2_n)_n$ of $(x_n)_n$ and some sequences $(y_n)_n$, $(z_n)_n$ in $\R$ 
such that
$$\liminf_{n\to\infty}(U(z_n,x^1_n+y_n)-U(z_n,x^2_n+y_n))>0.$$
By (\ref{eq-cvU}), $(z_n)_n$ is bounded.
Then, letting $\zeta$ be one of its limit points, the uniform continuity of $U$
yields
$$\liminf_{n\to\infty}(U(\zeta,x^1_n+y_n)-U(\zeta,x^2_n+y_n))>0.$$
For $i=1,2$ and $n\in\N$, define
$$u_n^i(t,x):=U\left(\int_0^x\sigma(x_n^i+y_n+s)ds-t,x_n^i+y_n+x\right).$$
We see that the $u_n^i$ satisfy (\ref{eqprinc}) with $a$ and $c$ 
translated by $x_n^i+y_n+x$, and that
$$\liminf_{n\to\infty}(u_n^1-u_n^2)(-\zeta,0)>0.$$
Moreover, $(a(\cdot+x_n^1+y_n))_n$, $(c(\cdot+x_n^1+y_n))_n$ and $(\sigma(\cdot+x_n^1+y_n))_n$
converge (up to subsequences) to some functions $a^*$, $c^*$
and $\sigma^*$ uniformly in $\R$. Since $(a(\cdot+x_n))_n$, $(c(\cdot+x_n))_n$ and
$(\sigma(\cdot+x_n))_n$ converge uniformly in $\R$, it turns out that also $(a(\cdot+x_n^2+y_n))_n$,
$(c(\cdot+x_n^2+y_n))_n$ and $(\sigma(\cdot+x_n^2+y_n))_n$
converge uniformly to $a^*$, $c^*$ and $\sigma^*$.
By parabolic estimates, we find that the $(u_n^i)_n$ converge (up to
subsequences) locally
uniformly to some functions $v^i$ which satisfy~(\ref{eqprinc}) with $a$ and $c$
replaced  by $a^*$ and $c^{*}$. We also have that 
$(v^1-v^2)(-\zeta,0)>0$. It is
straightforward to check that the functions $V^i$ defined by
$V^i(z,x):=v^i\left(\int_0^x\sigma^*-z,x\right)$ satisfy
$$\forall z,x\in\R,\quad V^i(z,x)=\lim_{n\to\infty}U(z,x_n^i+y_n+x),$$
and that they solve (\ref{deg}) with $a$, $c$ and $\sigma$ replaced by
$a^{*}$, $c^*$ and $\sigma^*$ respectively. 
Set
$$\kappa^{*}:= \sup_{\R\times\R}\frac{V^1}{V^2}.$$
The fact that both $V^1$ and $V^2$ inherit from $U$ the uniform convergence to 
$1$ as $z\to -\infty$, as well as the inequalities (\ref{eq:exp}), imply that
$V^1/V^2\to1$ as $z\to \pm\infty$ uniformly in $x\in \R$. 
Since they also inherit \eqref{eq:inf}, we infer that $\kappa^{*}$ is finite.
On the other hand,
we know that $\kappa^{*}>1$ because $(V^1-V^2)(\zeta,0)=(v^1-v^2)(-\zeta,0)>0$. 
As a consequence, using the uniform continuity of $V^1$ and $V^2$, we find that
$\kappa^{*}$ is attained at some finite $\bar z$, in the sense that
$$\sup_{x\in\R}\frac{V^1}{V^2}(\bar z,x)=\sup_{\R\times\R}\frac{V^1}{V^2}=
\kappa^{*}.$$
Direct computation shows that $\kappa^{*} V^{2}$ is a 
supersolution of (\ref{deg}), because $\kappa^{*}>1$. 
We can therefore apply Lemma \ref{lem:I}, 
deducing that
$$\forall z'>z, \quad  \inf_{x\in \R}(\kappa^{*}V^{2}-V^{1})(z',x)= 
0.$$
From this, taking $z'$ sufficiently large, we get a contradiction with (\ref{eq:exp}). 
\end{proof}

%---------------------------------------------------------------------------------

\subsection{Existence of fronts with supercritical speeds}

% \subsection{Generalized transition front property.}

We can now complete the proof of Theorem 
\ref{thm-apgen}(i) in the case $w>w^{*}$.

\begin{prop} \label{cor-ap}
For any $w\in (w^*,\underline{w})$, the solution 
$$u(t,x)=U\left(\frac1\gamma\int_0^x\sigma_\gamma(y) dy-t,x\right)$$ 
constructed before, with $\gamma$ given by 
Lemma~\ref{lem:decmu}, is a generalized transition front with 
average speed $\gamma/\langle \sigma_\gamma \rangle=w$.
\end{prop}

\begin{proof} 
% [Proof of Proposition \ref{cor-ap}]
%Firstly, the time-monotonicity of $u$ follows from the monotonicity in $z$ of
%$U$.
Recall that, because $\sigma_\gamma$ is a.p., the following limit exists 
uniformly in $a\in\R$:
\begin{equation}\label{uniformmean}
\langle \sigma_\gamma\rangle=\lim_{x\to\pm\infty} 
\frac{1}{x}\int_a^{a+x}\sigma_\gamma.
\end{equation}
By construction, c.f.~in particular \eqref{average} and 
Lemma~\ref{lem:decmu}, we have that
$$\langle \sigma_\gamma \rangle=\langle-\phi_\gamma'/\phi_\gamma\rangle
=\mu(\gamma)=\gamma/w.$$
Since $\gamma/w>0$, this allows us to define $X(t)$, for $t\in\R$, as the 
smallest value for which  
$\int_0^{X(t)} \sigma_\gamma=\gamma t$.
% 
% Hence, $\int_{X(t)}^{x+ X(t)} \sigma_\gamma(y)dy \sim \langle
% \sigma_\gamma\rangle x$ as $x\to +\infty$ uniformly with respect to $t\in\R$. 
% 
% For all $t\in \R$, we  This is possible because 
% $\sigma_\gamma$ has average $\gamma>0$.
One has
$$u(t,x+X(t)) = U \left(\frac1\gamma\int_{X(t)}^{X(t)+x}
\sigma_\gamma(y)dy\,,\,x+ X(t)\right),$$
whence, using \eqref{uniformmean} and \eqref{eq-cvU}, we eventually derive
% $$\lim_{x\to +\infty} u(t,x+X(t))=0,\quad\text{uniformly in }t\in\R.$$ The 
% other convergence in 
\eqref{limits}.

Lastly, one has $\int_{X(s)}^{X(t)} \sigma_\gamma =\gamma(t-s)$. Hence,
$\gamma|t-s| \leq \|\sigma_\gamma\|_\infty |X(t)-X(s)|$ for all $s,t\in
\R$, and therefore
$|X(t)-X(s)| \to +\infty$ as $t-s\to +\infty$. It follows that 
$$ \frac{\gamma(t-s)}{ X(t) - X(s)} = \frac{1}{ X(t) -
X(s)}\int_{X(s)}^{X(t)} \sigma_\gamma \to \langle \sigma_\gamma \rangle 
\quad\hbox{ as } t-s\to +\infty,$$
that is, $u$ has average speed $\gamma/\langle \sigma_\gamma \rangle=w$.
\end{proof}

% \begin{proof}[Proofs of Theorems \ref{thm:gtf} and \ref{thm-apgen} part (i), 
% when $w>w^{*}$]
% For all $w\in (w^{*}, \underline{w})$, we found $\gamma>\lambda_{1}(\L,\R)$ 
% (see Lemma \ref{lem:decmu}) associated with a profile $U=U(z,x)$ (Proposition 
% \ref{prop-existenceu}) 
% and $\sigma = \sigma_{\gamma}(x)/\gamma$, both a.p.~in $x$ (Lemma \ref{lem-ap} 
% and Proposition \ref{prop:apU}), 
% such that $\langle \sigma \rangle = 1/w$.
% Moreover, $u$ is a generalized transition front (Proposition \ref{cor-ap}).

%%-------------------------------------------------------------------------------
%
%
%\section{Properties of the generalized transition fronts} \label{sec:properties}

\subsection{Non-existence of fronts with speed less than $w^{*}$}

The following proposition immediately implies statements (ii) of Theorems 
\ref{thm:gtf} and \ref{thm-apgen}.

\begin{prop} \label{prop:nonexistence}
Let $u$ be a generalized transition front of equation (\ref{eqprinc}) and let 
$X$ be such that \eqref{limits} holds. Then
$$\forall s \in \R, \quad \liminf_{t\to+\infty} \frac{X(s+t)-X(s)}{t} \geq w^{*}.$$
In particular, there 
exists no generalized transition front with average speed $w<w^{*}$.
\end{prop}

\begin{proof}
For any $p\in \R$, we consider the second 
order elliptic operator with a.p.~coefficients
$$\L_{p}\phi:= e^{px}\mathcal{L} \big( e^{-px}\phi\big)$$
and we let $\underline{\lambda_{1}}
(\L_{p})$ and $\overline{\lambda_{1}} (\L_{p})$ denote the generalized 
principal eigenvalues defined by
(\ref{eq:lambda1}) and (\ref{eq:lambda12}) with $\mathcal{M}=\L_{p}$. We know 
from 
the proof of Lemma \ref{lem-zetatheta} that 
$\underline{\lambda_{1}} (\L_{p})=\overline{\lambda_{1}} (\L_{p})$, and we 
call this quantity $k_{p}$ in order to shorten the notation. 
It has been proved by Berestycki and the first author in \cite[Theorem 
2.3]{BN1} that if $v$ is a solution of 
equation 
(\ref{eqprinc}) associated with a continuous initial datum with compact support 
$v_{0}\geq0,\not\equiv0$, then 
$\lim_{t\to +\infty}v(t,wt)=1$ for all $w\in \big( 0, 
\min_{p>0}k_p/p\big)$. 

Consider now a generalized transition front $u$. Then there exists 
$\epsilon_{0}>0$ such that 
$$\forall s\in \R, \quad 2\epsilon_{0}< u\big(s,X(s)\big)<1-2\epsilon_{0},$$
see $(1.10)$ in \cite[Theorem 1.2]{BerestyckiHamelgtw2}.
Since $u$ is uniformly continuous by regularity estimates, there is 
$\ell>0$ such that $u\big(s,X(s)+x\big)>\epsilon_0$ for all $s\in 
\R$ and $|x|<\ell$. Take $s\in\R$. 
Comparing with the solution $v$ of equation (\ref{eqprinc}) for $t>s$,
associated with a datum at time $s\in\R$ supported in 
$\big(X(s)-\ell,X(s)+\ell\big)$ and smaller than $\epsilon_{0}$, we 
get that $u(s+t,x) \geq v(s+t,x)$ for all $t>0$, $x\in \R$, and thus 
$$\forall w\in \big(0, \min_{p>0}k_p/p\big),\quad
\lim_{t\to +\infty} u\big(s+t,X(s)+wt\big)=1.$$
The definition of generalized transition fronts then yields that there exists 
$L>0$ such that, for $t$ large enough (depending on $s$), $X(s)+wt\leq X(s+t)+L$.
This gives
$$\liminf_{t\to +\infty}\frac{X(s+t)-X(s)}{t} \geq w.$$
As this holds for all $w\in \big( 0, \min_{p>0}k_p/p\big)$, we get the 
conclusion providing we could show that $\min_{p>0}k_p/p = w^{*}$, where, we 
recall, $w^*:= \inf_{\gamma>\lambda_1}\gamma/\mu (\gamma)$ with
$\mu(\gamma)$ given by Proposition~\ref{prop-phigamma}.

It has also been proved in \cite{BN1} that $k_{p}\geq 
\lambda_{1}$ (Lemma 5.1), $k_{0}= \lambda_{1}$ (Step 2 of the proof of Theorem 5.1) and that 
$\mu:( \lambda_{1},+\infty)\to ( \underline{\mu},+\infty) 
$ is an homeomorphism with inverse $p\mapsto k_{p}$ (see the proof of 
\cite[Theorem 5.1]{BN1}).
Hence, we could rewrite
\begin{equation*} %\label{eq:caracw*2}
w^{*} = \inf_{\gamma> \lambda_{1}(\L,\R)}\frac{\gamma}{\mu (\gamma)}
= \inf_{p> \underline{\mu}}\frac{k_{p}}{
p}.
\end{equation*}
This concludes the proof if $\underline\mu=0$. Otherwise, since $k_{p}\geq 
\lambda_{1}$, $k_{0}= 
k_{\underline{\mu}}=\lambda_{1}$ and $p\mapsto k_{p}$ is convex, one has 
$k_{p} = \lambda_{1}$ for all $p\in (0,\underline{\mu})$. In particular, 
$k_{p}/p \geq \lambda_{1}/\underline{\mu}$ for $p \in (0,\underline{\mu})$. 
We thus conclude that 
$$w^{*} =  \inf_{p> \underline{\mu}}\frac{k_{p}}{p} =  \inf_{p>0}\frac{k_{p}}{p}.$$

\end{proof}

%-------------------------------------------------------------------------------

\subsection{The critical front} \label{sec:critical}

\begin{rmk}\label{rmk:smooth}
 We can always assume that the function $X$ for which a given front 
satisfies~\eqref{limits} is uniformly Lipschitz-continuous (and actually even 
$\mathcal{C}^\infty$). Indeed, we know from \cite[Proposition 4.2]{HR1} that
$$\sup_{(t,s)\in\R^2,\,|t-s|\leq1}|X(t)-X(s)|<+\infty.$$
Then, defining $\hat X$ as the linear interpolation of the function $X$ 
restricted to $\Z$, one readily sees that $\hat X\in 
W^{1,\infty}(\R)$ and that~\eqref{limits} holds true with $X$ replaced by $\hat
X$, because, denoting the integer part of $t\in\R$ by $[t]$, there holds
$$\sup_{t\in\R}|X(t)-\hat X(t)|\leq \sup_{t\in\R}|X(t)-X([t])|+
\sup_{t\in\R}|\hat X(t)-\hat X([t])|<+\infty.$$
Of course, one could use a smooth interpolation in order to get 
$\hat X\in \mathcal{C}^\infty(\R)$.
\end{rmk}

\begin{proof}[Proofs of Theorems \ref{thm:gtf} and \ref{thm-apgen} part (i) when $w=w^{*}$.]
We now assume that $w^{*}<\underline{w}$ and
consider the critical travelling wave $u$ associated with equation 
(\ref{eqprinc}), 
in the sense introduced by the first author in \cite{NadinCTW}, normalized by 
$u(0,0)=1/2$. Theorem 3.6 of \cite{NadinCTW} yields that $u$ is a generalized 
transition front (called spatial travelling wave in \cite{NadinCTW}), which is increasing with respect to $t$ (by Proposition 3.5 of \cite{NadinCTW}). It 
follows in particular that \eqref{limits} holds with $X:\R\to\R$ such 
that $u\big( t,X(t)\big)=1/2$ for all $t\in \R$. Moreover, the 
proof of this earlier result yields that if another generalized transition 
front satisfies \eqref{limits} with a function $Y$ instead of $X$, then 
there exists $L>0$ such 
that, for all $s_0\in \R$, one can find $s_1\in \R$ so that 
$$\forall t>0,\quad
X(s_{0}+t)-X(s_{0}) \leq Y(s_{1}+t) - Y(s_{1})+L.$$
Considering any generalized transition front with supercritical average 
speed $w\in(w^{*},\underline{w})$ constructed before, we thus obtain
$$\limsup_{t\to +\infty}\sup_{s_{0}\in\R}\frac{X(s_{0}+t)-X(s_{0})}{t} \leq 
\lim_{t\to +\infty}\sup_{s_{1}\in\R}\frac{Y(s_{1}+t) - Y(s_{1})+L}{t}=w.$$
This being true for $w$ arbitrarily close to $w^{*}$, we get 
\begin{equation} \label{X'<w*}
\limsup_{t\to +\infty}\sup_{s\in\R}\frac{X(s+t)-X(s)}{t} \leq w^{*}.
\end{equation}
By Remark \ref{rmk:smooth}, we know that there is a uniformly 
Lipschitz-continuous function, that we still call $X$, for which $u$ 
fulfils \eqref{limits}. Since this function is obtained as a bounded
perturbation of the 
previous $X$, \eqref{X'<w*} holds true (we just lose the information 
$u\big( t,X(t)\big)=1/2$). 
This allows us to rewrite \eqref{X'<w*} in terms of the {\em
upper mean} of $X'\in L^\infty(\R)$, which
is defined by
$$\um{X'}:=\lim_{t\to +\infty}\sup_{s\in\R}\int_s^{s+t}X'(\tau)\,d\tau.$$
Namely, we have $\um{X'}\leq w^{*}$.
The notion of upper mean, together with that of {\em least mean}:
$$\lm{X'}:=\lim_{t\to +\infty}\inf_{s\in\R}\int_s^{s+t}X'(\tau)\,d\tau,$$
has been introduced in \cite{NR1}. Actually, the above formulation for the 
least mean - from which the analogous one for the upper mean immediately 
follows - is not the original definition of~\cite{NR1}, but comes from 
Proposition 3.1 there, which 
shows in particular the existence 
of the limit. Clearly, $u$ has an average speed $w$ if and only if 
$\lm{X'}=\um{X'}=w$. Thus, in order to conclude the proof, we need to
show that $\lm{X'}\geq w^*$. 

To achieve our goal, we make use of another 
characterization of the least mean, provided by \cite[Proposition 4.4]{NR2}, 
which involves the $\omega$-limit set of the function. It implies the existence
of a sequence $(s_n)_{n}$ such that $X'(\cdot+s_n)$ converges 
weakly-$\star$ in $L^\infty(\R)$ to a function $g$ satisfying
$$\lim_{t\to+\infty}\frac{1}{t}\int_0^t g(\tau)\, d\tau=\lm{X'}.$$
Define the sequence of functions $(u_n)_{n}$ by
$$u_n(t,x):=u(t+s_n,x+X(s_n)).$$
It follows from interior parabolic estimates that $(u_n)_{n}$ converges 
locally uniformly  to a function $\tilde u$ satisfying 
\eqref{limits} with $X(t)=\int_0^t g$. Moreover, $a(\cdot+X(s_n))$, 
$a'(\cdot+X(s_n))$, $c(\cdot+X(s_n))$ converge uniformly in $\R$ (up to 
subsequences) to some functions $\tilde a$, $b=(\tilde a)'$, $\tilde c$ and 
therefore $\tilde u$ is a generalized transition front of the equation
$$
 v_t - \big(\tilde a(x)v_x\big)_x = \tilde c(x) v(1-v), \quad t\in\R,\ x\in\R.
$$
This equation fulfils the same set of standing assumptions as 
\eqref{eqprinc}. In particular, Proposition~\ref{prop-phigamma} holds 
true for the operator $\phi\mapsto\tilde\L \phi := \big(\tilde a(x) \phi'\big)' 
+ \tilde c(x) \phi$ and, for all $\gamma>\lambda_{1}(\tilde\L,\R)$, provides 
us with a unique positive solution to
$$
\tilde\L\tilde\phi_{\gamma}= \gamma \tilde\phi_{\gamma} \ \hbox{
in } \R, \quad \tilde\phi_{\gamma}(0)=1, \quad \lim_{x\to +\infty} 
\tilde\phi_{\gamma}(x)=0,
$$
which satisfies in addition $\tilde\mu (\gamma) := 
-\lim_{x\to +\infty} \frac{1}{x}\ln \tilde\phi_\gamma (x)>0$.

Let us show that $\lambda_1 (\tilde\L,\R)=\lambda_1(\L,\R)$ and that 
$\tilde\mu\equiv\mu$.
First, the operator $\tilde\L$ is a limit operator associated
with $\mathcal{L}$ in the sense of \cite{BerestyckiHamelRossi}, and therefore 
Lemma 5.6 there yields $\lambda_1 (\tilde\L,\R)=\lambda_1(\L,\R)$.
Next, consider the solution $\phi_\gamma$ provided by 
Proposition~\ref{prop-phigamma}. The sequence 
$\phi_\gamma(\cdot+X(s_n))/\phi_\gamma(X(s_n))$ converges in 
$\mathcal{C}^2_{loc}(\R)$ to a nonnegative function $\psi$ satisfying 
$\tilde\L\psi= \gamma \psi$ in $\R$ and $\psi(0)=1$. It follows from the strong 
maximum principle that $\psi$ is positive. We then compute
$$\lim_{x\to+\infty} \frac{1}{x}\ln 
\psi(x)=\lim_{x\to+\infty}\frac{1}{x}\int_0^x\frac{\psi'}\psi=
\lim_{x\to+\infty}\lim_{n\to\infty}\frac{1}{x}\int_{X(s_n)}^{X(s_n)+x}
\frac{\phi_\gamma'}{\phi_\gamma}=-\mu(\gamma),$$
where we have used that the a.p.~function $\phi_\gamma'/\phi_\gamma$ 
satisfies \eqref{average} uniformly in $z\in\R$.
Since $\mu(\gamma)>0$, this shows that $\psi$ decays to $0$ at $+\infty$ and 
therefore, by uniqueness, it coincides with $\tilde\phi_{\gamma}$. We 
eventually 
infer that $\tilde\mu(\gamma)=\mu(\gamma)$.
%In particular, as $w^{*}<\underline{w}$, one has 
%$$\inf_{\gamma>\lambda_{1}(\tilde{\L},\R)}\frac{\gamma}{\tilde{\mu} (\gamma)} < \frac{\lambda_{1}(\tilde{\L},\R)}{\underline{\mu}}.$$\nota{Du coup j'ai ajoute ca}
We can then apply Proposition~\ref{prop:nonexistence} to the front $\tilde u$ 
and obtain
$$\inf_{\gamma>\lambda_1 (\tilde\L,\R)} \displaystyle\frac{\gamma}{\tilde\mu 
(\gamma)}=w^{*}\leq
\liminf_{t\to+\infty} \frac{\int_0^{t}g(\tau)\,d\tau}{t}=\lm{X'}.$$
This concludes the proof. 
\end{proof}

%-------------------------------

\subsection{Proof of Lemma \ref{lem:nonempty}.}

\begin{proof}[Proof of Lemma \ref{lem:nonempty}]
For all $c_{0}>-\inf c$, let $\mu_{c_{0}}(\gamma)$ be 
the decay rate provided by Proposition \ref{prop-phigamma}(ii) but associated 
with 
$c(x)+c_{0}$ instead of $c (x)$. 
It is straightforward to see that $\lambda_{1}(\L+c_{0},\R)= 
\lambda_{1}(\L,\R)+c_{0}$ and that 
$\phi_{\gamma}$ satisfies
$$\big( a(x) \phi_{\gamma}'\big)'+\big(c(x)+c_{0}\big) \phi_{\gamma}= (\gamma+c_{0})\phi_{\gamma} \quad \hbox{in } \R.$$
Hence, it immediately follows that $\mu_{c_{0}}(\gamma+c_{0}) = \mu (\gamma)$ 
and then
that $\underline{\mu}$ does not change when one adds $c_0$ to $c(x)$. 

Now take $\gamma>\lambda_{1}(\L,\R)$. As $\mu (\gamma)>\underline{\mu}$ by 
Lemma \ref{lem-convexitymu}, for 
$c_{0}$ large enough we have that 
$$\mu (\gamma) \big( \lambda_{1}(\L,\R)+c_{0}\big)> \underline{\mu}\big(\gamma 
+ c_{0}).$$
In other words, setting $\tilde{\gamma}:= \gamma +c_{0}$, one gets
$$\frac{\mu_{c_{0}}(\tilde{\gamma})}{\tilde{\gamma}} > \frac{\underline{\mu}}{\lambda_{1}(\L+c_{0},\R)},$$
from which the result follows by taking the supremum over $\tilde{\gamma}> 
\lambda_{1}(\L+c_{0},\R)$.
\end{proof}

%%%%%%%%%%%%%%%%%%%%%%%%%%%%%%%%%%%%%%%%%%%%%%%%%%%%%%%%%%%%%%%%%%%%%%%%%%%%%%%%

\section*{Compliance with Ethical Standards}

The authors declare that there are no conflicts of interest.

%%%%%%%%%%%%%%%%%%%%%%%%%%%%%%%%%%%%%%%%%%%%%%%%%%%%%%%%%%%%%%%%%%%%%%%%%%%%%%%%

\end{document}